\documentclass[12pt]{amsart}
\usepackage{graphicx,color}
\usepackage[latin1]{inputenc}
\usepackage[english]{babel}
\usepackage{amsmath,amssymb,amsfonts,amsthm,mathrsfs}

\usepackage{graphicx, subfigure,color} 
\usepackage{epstopdf}

\theoremstyle{plain}

\theoremstyle{plain}
\newtheorem{theorem}{Theorem} [section]

\newtheorem{lemma}[theorem]{Lemma}
\newtheorem{proposition}[theorem]{Proposition}

\theoremstyle{definition}
\newtheorem{definition}[theorem]{Definition}
\newtheorem{example}[theorem]{Example}
\newtheorem{remark}[theorem]{Remark}

\numberwithin{theorem}{section}
\numberwithin{equation}{section}
\numberwithin{figure}{section}

\def\mean#1{\mathchoice
         {\mathop{\kern 0.2em\vrule width 0.6em height 0.69678ex depth -0.58065ex
                 \kern -0.8em \intop}\nolimits_{\kern -0.4em#1}}%
         {\mathop{\kern 0.1em\vrule width 0.5em height 0.69678ex depth -0.60387ex
                 \kern -0.6em \intop}\nolimits_{#1}}%
         {\mathop{\kern 0.1em\vrule width 0.5em height 0.69678ex
             depth -0.60387ex
                 \kern -0.6em \intop}\nolimits_{#1}}%
         {\mathop{\kern 0.1em\vrule width 0.5em height 0.69678ex depth -0.60387ex
                 \kern -0.6em \intop}\nolimits_{#1}}}

 \renewcommand{\k}{\kappa}
 \newcommand{\R}{\mathbb{R}}
 \newcommand{\T}{\mathbb{T}}

 \newcommand{\N}{\mathbb{N}}
 \newcommand{\Z}{\mathbb{Z}}
 \newcommand{\C}{\mathbb{C}}

 \renewcommand{\H}{\mathbb{H}}
 \def\SP{\mathbb S}
 \def\TCL{{\rm TCL}}

 \newcommand{\DD}{\mathcal{D}}

 \newcommand{\uu}{{\mbox{\boldmath$u$}}}

 \newcommand{\tauV}{{\kern-3pt\tau}}

 \newcommand{\oVVVk}{\overline{\mbox{\boldmath$V$}}\kern-3pt}
 \newcommand{\tVVVk}{\tilde{\mbox{\boldmath$V$}}\kern-3pt}

 \newcommand{\res}{\mathop{\hbox{\vrule height 7pt width .5pt depth 0pt
 \vrule height .5pt width 6pt depth 0pt}}\nolimits}

  \newcommand{\Id}{{\rm Id}}

 \newcommand{\cut}{{\rm cut}}

 \def\Smtw{\mathfrak{S}}
 \def\MTW{{\rm MTW}}

 \def\I{{\rm I}}

 \newcommand{\<}{\langle}
 \renewcommand{\>}{\rangle}
 \renewcommand{\a}{\alpha}
 \renewcommand{\b}{\beta}

 \newcommand{\e}{\varepsilon}
 \newcommand{\g}{\gamma}

 \newcommand{\n}{\nabla}
 \newcommand{\var}{\varphi}

 \renewcommand{\i}{\infty}
 \newcommand{\p}{\partial}

 \renewcommand{\det}{\operatorname{det}}
 \newcommand{\dist}{\operatorname{dist}}
 \newcommand{\diam}{\operatorname{diam}}
 \newcommand{\supp}{\operatorname{supp}}
 \newcommand{\osc}{\operatorname{osc}}

\DeclareMathOperator{\cexp}{c-exp}

\date{}
\title[The Monge-Amp\`ere equation
and its link to optimal transportation]
{The Monge-Amp\`ere equation\\
and its link to optimal transportation}

\author[G. De Philippis]{Guido De Philippis}
\address{Hausdorff Center for Mathematics,
Universit\"at Bonn,
Endenicher Allee 60,
53115 Bonn, Germany}
\email{guido.de.philippis@hcm.uni-bonn.de}

\author[A. Figalli]{Alessio Figalli}
\address{The University of Texas at Austin,
Mathematics Dept. RLM 8.100,
2515 Speedway Stop C1200,
Austin, Texas 78712-1202, USA}
\email{figalli@math.utexas.edu}

\begin{document}
\maketitle

\begin{abstract}
We survey the (old and new) regularity theory for the Monge-Amp\`ere 
equation, show its connection to optimal transportation, and describe the regularity
properties of a general class of Monge-Amp\`ere type equations arising in that context.
\end{abstract}

\tableofcontents

\medskip    \section{Introduction}

The Monge-Amp\`ere equation is a fully nonlinear degenerate elliptic 
equation which arises in several problems from analysis and geometry.
In its classical form this equation is given by
\begin{equation}
\label{eq:MAclassical}
\det D^2u=f(x,u,\n u) \qquad \text{in $\Omega$},
\end{equation}
where $\Omega\subset \R^n$ is some open set,
$u:\Omega \to \R$ is a convex function,
and $f:\Omega\times \R\times \R^n\to \R^+$ is given.
In other words, the Monge-Amp\`ere equation prescribes the 
product of the eigenvalues of the Hessian of $u$, in contrast with the ``model'' elliptic equation
$\Delta u=f$ which prescribes their sum.
As we shall explain later, the convexity of the solution $u$ is a necessary condition to make the equation degenerate elliptic, and so to hope for regularity results.\\

The prototype equation where the Monge-Amp\`ere equation appears is the ``prescribed Gaussian curvature equation": if we take $f=K(x)\bigl(1+|\n u|^2 \bigr)^{(n+2)/2}$ then
\eqref{eq:MAclassical} corresponds to imposing that the Gaussian curvature of the graph of $u$ 
at the point $(x,u(x))$ is equal to $K(x)$. 
Another classical instance where the Monge-Amp\`ere equation arises is affine geometry,
more precisely
in the ``affine sphere problem'' and the ``affine maximal surfaces'' problem (see for instance \cite{cal,pog affine,cy,TW1aff,TW2aff,TW3aff}). The Monge-Amp\`ere equation \eqref{eq:MAclassical} also arises in meteorology and fluid mechanics:
for instance, in the semi-geostrophic equations
it is coupled with a transport equation (see Section \ref{semigeo} below).

As we shall see later, in the optimal transportation problem  the study of Monge-Amp\`ere type equation of the form
\begin{equation}
\label{eq:MAgeneral}
\det\bigl(D^2u - \mathcal A(x,\n u)\bigr)=f(x,u,\n u),
\end{equation}
plays a key role in understanding the regularity (or singularity) of optimal transport maps.

More in general,  Monge-Amp\`ere type equations
of the form 
\begin{equation*}
\det\bigl(D^2u - \mathcal A(x,u,\n u)\bigr)=f(x,u,\n u),
\end{equation*}
have found applications in several other problems, such as
isometric embeddings, reflector shape design,
and in the study of special Lagrangian sub-manifolds, prescribing Weingarten curvature, and in complex geometry on toric manifolds (see for instance the survey paper \cite{TWsurvey}
for more details on
  these geometric applications). \\

The aim of this article is to describe the general regularity theory for the Monge-Amp\`ere equation
\eqref{eq:MAclassical}, show its connections with the optimal transport problem
with quadratic cost,
introduce the more general class of Monge-Amp\`ere type equations \eqref{eq:MAgeneral}
which arise when the cost is not quadratic, and show 
some regularity results for this class of equations.

\medskip    \section{The classical Monge-Amp\`ere equation}
\label{sect:classical MA}
The Monge-Amp\`ere equation \eqref{eq:MAclassical} draws its name from its initial formulation in two dimensions, by the French mathematicians Monge \cite{Mon2} and Amp\`ere \cite{Am}, about two hundred years ago.
As we mentioned above, this equation is degenerate elliptic on convex functions. Before describing some history of the regularity theory for this equation, let us first explain this point.

\begin{remark}[{On the degenerate ellipticity of the Monge-Amp\`ere equation}]
\label{rmk:degen}
{\rm Let $u:\Omega \to \R$ be a smooth solution of \eqref{eq:MAclassical} with $f=f(x)>0$ smooth, and let us try to understand if we can prove some regularity estimates.
A standard technique consists in differentiating the equation solved by $u$ to obtain an equation for its
first derivatives. Hence we fix a direction $e \in \mathbb S^{n-1}$ and we differentiate \eqref{eq:MAclassical} in the direction $e$ to obtain
$$
\det( D^2 u ) \,u^{ij}\p_{ij} u_e=f_e \qquad \text{in $\Omega$},
$$
where $u^{ij}$ denotes the inverse matrix of $u_{ij}:=(D^2u)_{ij}$, lower indices 
denotes partial derivatives  (thus $u_e:=\p_eu)$, and we are summing over repeated indices. Recalling that $\det D^2 u =f>0$ by assumption, we can rewrite the above equation as
\begin{equation}
\label{eq:linearMA}
u^{ij} \p_{ij} u_e=\frac{f_e}{f} \qquad \text{in $\Omega$}.
\end{equation}
Hence, to obtain some regularity estimates on $u_e$ we would like the matrix $u^{ij}$
to be nonnegative definite (and if possible even positive definite) to apply elliptic regularity theory.
But for the matrix $u^{ij}$ to be nonnegative definite we need $D^2u$ to be nonnegative definite,
which is exactly the convexity assumption on $u$.
We now observe that, without any a priori bound on $D^2u$, then $u^{ij}$ may have arbitrarily small eigenvalues
and this is why we say that the equation is ``degenerate elliptic''.
However, if one can show that ${\rm Id}/C \leq D^2u \leq C{\rm Id}$ inside $\Omega$
for some constant $C>0$,
then ${\rm Id}/C \leq u^{ij} \leq C{\rm Id}$ and the linearized equation \eqref{eq:linearMA}
becomes uniformly elliptic. For this reason the bound ${\rm Id}/C \leq D^2u \leq C{\rm Id}$ 
is one of the key steps for the regularity of solutions to \eqref{eq:MAclassical}.
In this regard we notice that, if we assume that $f>0$ is uniformly bounded away from zero, then the product of the eigenvalues of $D^2u$ is bounded away from zero 
and to obtain the estimate ${\rm Id}/C \leq D^2u \leq C{\rm Id}$ it is actually enough to prove only the upper bound $|D^2u|\leq \bar C$ for some constant $\bar C$.}
\end{remark}

We now give a brief overview on the existence and regularity theory for the Monge-Amp\`ere equation.

The first notable results are by Minkowski \cite{M1,M2} who proved the existence of a weak solution to the ``prescribed Gaussian curvature equation" (now called ``Minkowski problem'') by approximation by convex polyhedra with given face areas.
Using convex polyhedra with given generalized curvatures at the vertices, Alexandrov also proved the existence of a weak solution
in all dimensions, as well as the $C^1$ smoothness of solutions in two dimensions \cite{AL1,AL2,AL3}.

In high dimensions, based on his earlier works, Alexandrov \cite{AL4} (and also Bakelman \cite{Bak1} in two dimensions) introduced
a notion of generalized solution to the Monge-Amp\`ere equation and proved the existence and uniqueness of solutions to the Dirichlet problem (see Section \ref{sect:Alex sol}).
The treatment also lead to the Alexandrov-Bakelman maximum principle which plays a fundamental role in the study of non-divergence elliptic equations (see for instance \cite[Section 9.8]{GT}).
As we shall see in Section \ref{sect:Alex sol}, the notion of weak solutions introduced by Alexandrov
(now called ``Alexandrov solutions'') has  continued to be frequently used in recent years,
and a lot of attention has been drawn to prove smoothness of
Alexandrov solutions under suitable assumptions on the right hand side 
and the boundary data.

The regularity of generalized solutions in high dimensions is a very delicate problem.
Pogorelov found a convex function which is not of class $C^2$ but satisfies the
Monge-Amp\`ere equation  \eqref{eq:MAclassical} inside $B_{1/2}$ with positive analytic right hand side (see \eqref{eq:example pogorelov} below).
As we shall describe in Section \ref{sect:Caff Alex}, the main issue in the lack of regularity is the presence of 
a line segment in the graph of $u$.
Indeed, Calabi \cite{cal smooth} and Pogorelov \cite{pog smooth} were able to prove a priori 
 interior second derivative estimate for strictly convex solutions, or for solutions which do not contain a line segment with both endpoints on boundary.
 By the interior regularity theory for fully nonlinear uniformly elliptic equations established by Evans \cite{Ev} and Krylov \cite{Kry} in the 80's,
 Pogorelov's second derivative estimate implies the smoothness of strictly convex Alexandrov's generalized solutions.
 
By the regularity theory developed by Ivochkina \cite{Ivo},
Krylov
\cite{Kry2}, and
Caffarelli-Nirenberg-Spruck \cite{CNS}, using the continuity method (see Section \ref{sect:continuity}
for a description of this method) one obtains globally smooth solutions to the Dirichlet problem.
In particular,
 Alexandrov's  solutions are smooth up to the boundary provided all given data are smooth. 
 
In all the situations mentioned above, one assumes that $f$ is positive and sufficiently smooth. When $f$ is merely 
bounded away from zero and infinity, Caffarelli proved the $C^{1,\alpha}$ regularity of strictly convex solutions \cite{Caf2}.
Furthermore, when $f$ is continuous (resp. $C^{0,\alpha}$), Caffarelli proved by a perturbation argument interior $W^{2,p}$-estimate for any $p > 1$ (resp. $C^{2,\alpha}$ interior estimates)
\cite{Cafw2p}.
More recently, the authors proved interior $L\log L$ estimates on $D^2u$ when $f$ is merely 
bounded away from zero and infinity \cite{DF1}, and together with Savin they improved this result showing that
$u \in W^{2,1+\e}_{\rm loc}$ \cite{DFS}.

In the next sections we will give a description of these results.

\medskip    \subsection{Alexandrov solutions and regularity results}\label{sect:Alex sol} In his study of the Minkowski problem,
Alexandrov introduced a notion of weak solutions to the Monge-Amp\`ere equation allowing him to give a meaning to the Gaussian curvature of non-smooth convex sets.

Let us first recall that, given an open convex domain $\Omega$,
 the subdifferential of a convex function \(u:\Omega \to \R\) is given by
 \[
\partial u (x):=\{p\in \R^{n}\,:\, u(y)\ge u(x)+p\cdot (y-x)\quad \forall\, y \in \Omega\}.
\]
Then, one defines the \emph{Monge-Amp\`ere measure} of \(u\) as follows:
\begin{equation}\label{ch1:eq:mameasure}
\mu_{u}(E):=|\partial u (E)| \qquad \text{for every Borel set $E\subset \Omega$,}
\end{equation}
where 
$$
\partial u (E):=\bigcup_{x\in E}\partial u(x)
$$
and $|\cdot|$ denotes the Lebesgue measure.
It is possible to  show that the restriction of \(\mu_{u}\) to the Borel \(\sigma\)-algebra is actually a measure (see \cite[Theorem 1.1.13]{G}). Note that,
in case \(u\in C^{2}(\Omega)\), the change of variable formula gives
$$
|\partial u (E)|=|\nabla u(E)|=\int_E \det D^{2} u (x)\,dx\qquad \text{for every Borel set $E\subset \Omega$,}
$$
therefore
\[
\mu_{u}=\det D^{2} u(x)\,dx\qquad \text{inside $\Omega$.}
\]

\begin{example}\label{ex:disc}Let \(u(x)=|x|^{2}/2+|x_{1}|\), then (writing \(x=(x_{1},x') \in \R\times \R^{n-1}\))
\[
\partial u (x)=
\begin{cases}
\{x+e_{1}\}\quad &\text{if \(x_{1}> 0\)}\\
\{x-e_{1}\}\quad &\text{if \(x_{1}<0\)}\\
\{(t,x')  :\ |t|\le 1\} &\text{if \(x_{1}=0\)}.
\end{cases}
\]
Thus \(\mu_{u}=dx+\mathcal H^{n-1}\res \{x_{1}=0\}\),
where $\mathcal H^{n-1}$ denotes the $(n-1)$-dimensional Hausdorff measure.
\end{example}


\begin{definition}[Alexandrov solutions]\label{ch1:def:alesol}Given an open convex set \(\Omega\) and a Borel measure \(\mu\) on \(\Omega\), a convex
function \(u:\Omega \to \R\) is called an \emph{Alexandrov solution} to the Monge-Amp\`ere equation
\[
\det D^{2} u =\mu,
\]
if \(\mu=\mu_{u}\) as Borel measures.
\end{definition}

When \(\mu=f\,dx\) we will simply say that \(u\) solves 
\begin{equation}\label{maale}
\det D^2 u=f.
\end{equation}
In the same way, when we write  \(\det D^2 u \ge  \lambda \ (\le 1/\lambda)\) we mean that \(\mu_u \ge \lambda \,dx \ (\le 1/\lambda\,dx)\).

One nice property of the Monge-Amp\`ere measure is that it is stable under uniform convergence
(see \cite[Lemma 1.2.3]{G}):

\begin{proposition}\label{uni}
Let $u_k:\Omega \to \R$ be a sequence of convex functions converging locally uniformly to $u$.
Then the associated Monge-Amp\`ere measures $\mu_{u_k}$ weakly$^\ast$
converge to $\mu_u$ (i.e., in duality with
the space of continuous functions compactly supported in $\Omega$).
\end{proposition}

We now describe how to prove existence/uniqueness/stability of Alexandrov solutions
for the Dirichlet problem.

As we shall see, the existence of weak solution is proved by an approximation and Perron-type argument, and for this it is actually useful to know a priori that solutions, if they exist, are unique and stable.
We begin with the following simple lemma:

\begin{lemma}\label{ch2:lem:inclsottdiff}Let \(u\) and \(v\) be convex functions in \(\R^n\). If \(E\) is an open and bounded set such that \(u=v\) on \(\partial E\) and \(u\le v\) in \(E\), then
\begin{equation}\label{ch2:eq:inclsottdiff}
\partial u (E)\supset \partial v (E).
\end{equation}
In particular \(\mu_{u}(E)\ge \mu_{v}(E)\).
\end{lemma}

\begin{proof}Let \(p\in \partial v(x)\) for some \(x\in U\). Geometrically, this means that the plane
\[
y\mapsto v(x)+p\cdot(y-x)
\]
is a supporting plane to  \(v\) at \(x\), that is, it touches from below
the graph of $v$ at the point $(x,v(x))$. Moving this plane down until it lies below $u$ and then lifting it up until it touches the graph of \(u\) for the first time, we see that, for some constant \(a\le v(x)\),
\[
y\mapsto a+p\cdot(y-x)
\]
is a supporting plane to  \(u\) at some point \(\bar x\in \overline E\), see Figure \ref{touching}.

\begin{figure}[tbp]
\begin{center}
\includegraphics[scale=0.65]{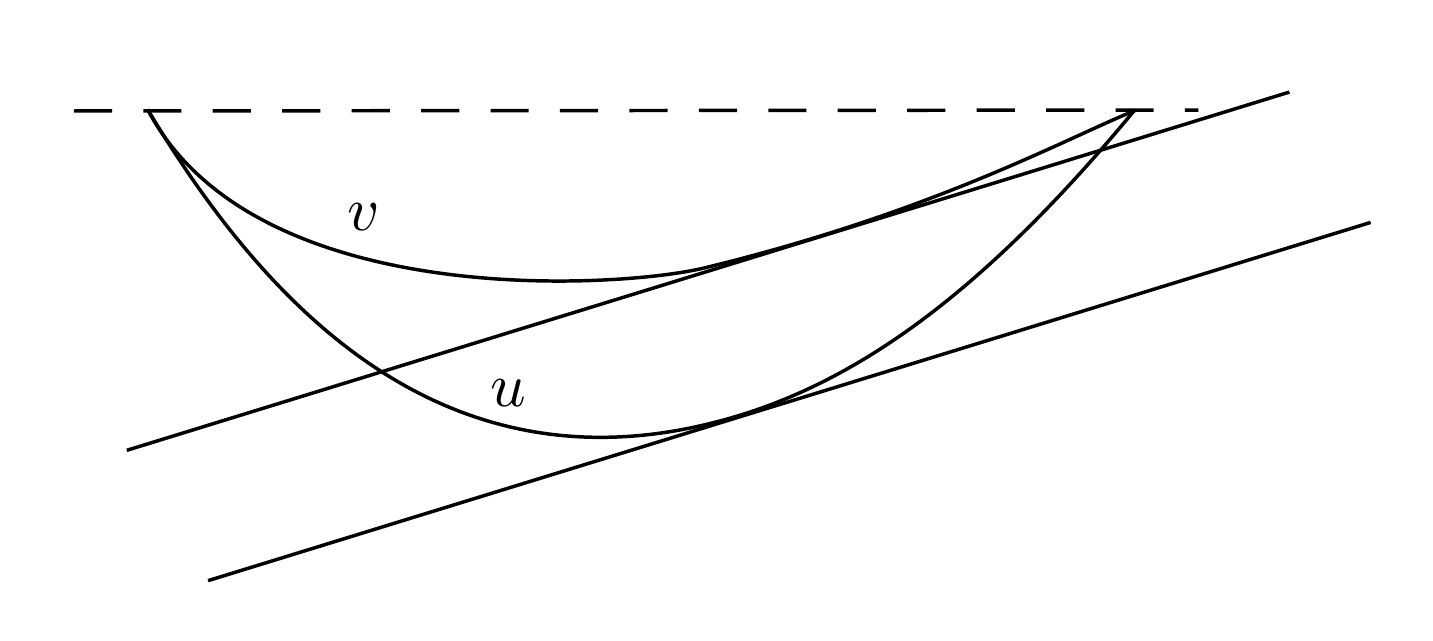}
\caption{Moving down a supporting plane to \(v\) until it lies below the graph of  $u$, and then lifting it up until it touches \(u\), we obtain a supporting place to \(u\) at some point inside \( E\).}
\label{touching}
\end{center}
\end{figure}

Since \(u=v\) on \(\partial E\) we see that, if \(\bar x\in \partial E\), then \(a=v(x)\) and thus \(u(x)=v(x)\) and the plane is also supporting \(u\) at \(x \in E\). In conclusion \(p\in \partial u (E)\), proving the inclusion \eqref{ch2:eq:inclsottdiff}.
\end{proof}

A first corollary is the celebrated Alexandrov's maximum principle:

\begin{theorem}[Alexandrov's maximum principle]\label{ch2:tim:alemax}Let \(u:\Omega \to \R\) be a convex function defined on an open, bounded and convex domain \(\Omega\). If \(u=0\) on \(\partial \Omega\), then
\[
|u(x)|^{n}\le C_{n}(\diam \Omega)^{n-1}\dist(x,\partial \Omega)|\partial u(\Omega)|\qquad \forall x\in\Omega,
\]
where \(C_{n}\) is a geometric constant depending only on the dimension.
 \end{theorem}
 
\begin{proof}Let \((x,u(x))\) be a point on the graph of \(u\) and let us consider the cone \(C_{x}(y)\) with vertex on \((x,u(x))\) and base \(\Omega\), that is, the graph of  one-homogeneous function (with respect to dilatation with center \(x\)) which is  \(0\) on \(\partial \Omega\) and equal to \(u(x)\) at \(x\). Since by convexity \(u(y)\le C_{x}(y)\), Lemma \ref{ch2:lem:inclsottdiff} implies
\[
|\partial C_{x}(x)|\le |\partial C_{x}(\Omega)|\le |\partial u(\Omega)|.
\]
(Actually, as a consequence of the argument below, one sees that $\partial C_x(x)=\partial C_x(\Omega)$.)
To conclude the proof we  have only to show that
\[
|\partial C_{x}(x)|\ge \frac{|u(x)|^{n}}{ C_{n}(\diam \Omega)^{n-1}\dist(x,\partial \Omega)}
\]
for some dimensional constant $C_n>0$.
Take \(p\) with \(|p|<|u(x)|/\diam \Omega\), and  consider a plane with slope \(p\).
By first moving it down and lifting it up until it touches the graph of \(C_{x}\), we see that it has to be supporting at some point \(\bar y\in \Omega\). Since \(C_{x}\) is a cone it also has to be supporting at \(x\). This means
\[
\partial C_{x}(x)\supset B(0,|u(x)|/\diam \Omega).
\] 
Let now \(\bar x\in \partial \Omega\) be such that \(\dist (x,\partial \Omega)=|x-\bar x|\) and let \(q\) be a vector with the same direction of \((\bar x-x)\) and with modulus less than \(|u(x)|/\dist (x,\partial \Omega)\). Then the plane \(u(x)+q\cdot (y-x)\) will be supporting \(C_{x}\) at \(x\) (see Figure \ref{ale2}), that is 
\[
q:= \frac{\bar x -x }{|\bar x-x|}\frac{|u(x)|}{|\dist (x,\partial \Omega)}\in \partial C_{x}(x).
\]
\begin{figure}
\begin{minipage}[t]{0.48\textwidth}
\includegraphics[scale=0.6]{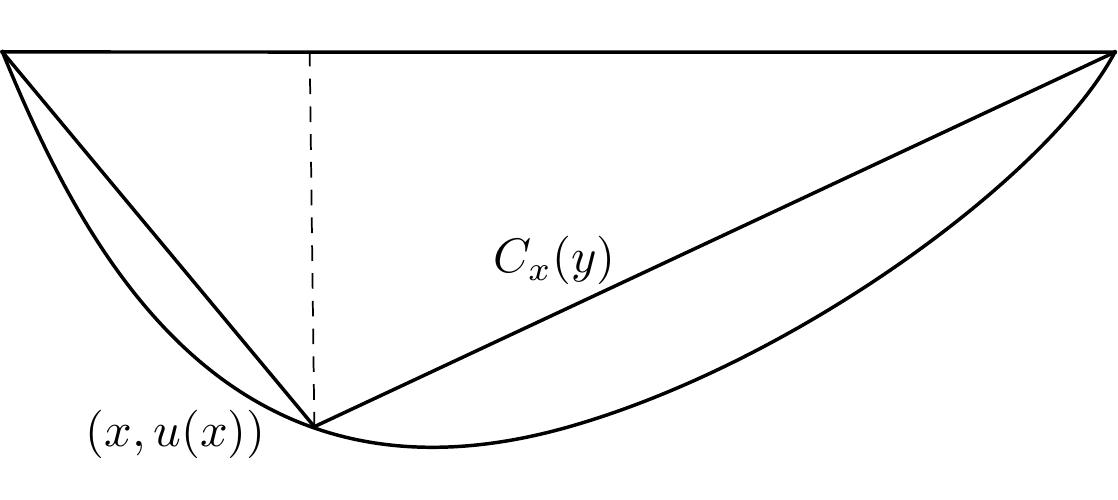}
\end{minipage}
\hfill
\begin{minipage}[t]{0.48\textwidth}
\includegraphics[scale=0.6]{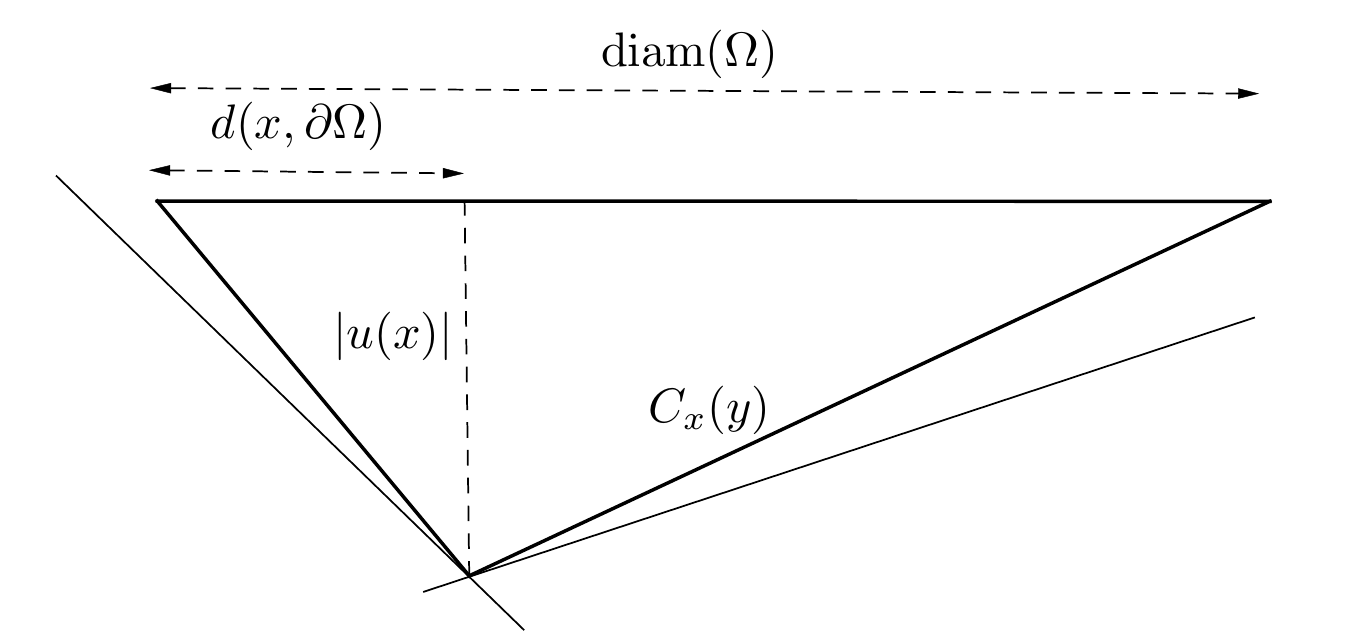}
\end{minipage}
\caption{\small{Every plane with slope \(|p|\le |u(x)|/\diam(\Omega)\) supports the graph of \(C_x\) at \(x\). Moreover there exists a  supporting plane whose slope has size comparable to $|u(x)|/\dist(x,\partial\Omega)$.}}
\label{ale2}
\end{figure}
By the convexity of \(\partial C_{x}(x)\) we have  that it contains the cone \(\mathcal C\) generated by \(q\) and \(B(0,|u(x)|/\diam \Omega)\). Since
\[
|\mathcal C| \geq \frac{|u(x)|^{n}}{ C_n(\diam \Omega)^{n-1}\dist(x,\partial \Omega)},
\] 
this concludes the proof.
\end{proof}

Another consequence of Lemma \ref{ch2:lem:inclsottdiff} is the
following comparison principle:

\begin{lemma}\label{ch2:lem:comparison}
Let \(u,v\) be convex functions defined on an open bounded convex set \(\Omega\). If \(u\ge v\) on \(\partial \Omega\) and (in the sense of Monge-Amp\`ere measures)
\[
\det D^{2}u \le \det D^{2}v\quad \text{in \(\Omega\)}, 
\]
then \(u\ge v\) in \(\Omega\).
\end{lemma}

\begin{proof}[Sketch of the proof]
Up to replacing $v$ by $v+\e(|x-x_0|^2-{\rm diam}(\Omega)^2)$ where $x_0$
is an arbitrary point in $\Omega$ and then letting $\e \to 0$, we can assume that
$\det D^2u<\det D^2 v$.

The idea of the proof is simple: if $E:=\{u<v\}$ is nonempty, then 
we can apply Lemma \ref{ch2:lem:inclsottdiff} to obtain
$$
\int_E \det D^2 u=\mu_u(E) \geq \mu_v(E)=\int_E \det D^2 v.
$$
This is in contradiction with  $\det D^2u<\det D^2 v$ and concludes the proof.
\end{proof}

A immediate corollary of the comparison principle is the uniqueness for Alexandrov solutions
of the Dirichlet problem.
We now actually state a stronger result concerning the stability of solutions,
and we refer to \cite[Lemma 5.3.1]{G} for a proof.

\begin{theorem}
\label{thm:stability}
Let $\Omega_k\subset \R^n$ be a  family of convex domains, and  let $u_k:\Omega_k \to \R$ be convex Alexandrov solutions of
\begin{equation*}
\label{MA_k}
\begin{cases}
\det D^2 u_k=\mu_k \quad &\text{in $\Omega_k$}\\
u_k=0 &\text{on $\partial \Omega_k$},
\end{cases}
\end{equation*}
where
$\Omega_k$ converge to some convex domain $\Omega$
in the Hausdorff distance, and 
$\mu_k$ is a sequence of 
nonnegative Borel measures with $\sup_k\mu_k(\Omega_k)<\infty$ and which
converge weakly$^\ast$ to a Borel measure $\mu$.
Then $u_k$ converge uniformly to the Alexandrov solution of
\[
\begin{cases}
\det D^2 u=\mu \quad &\text{in $\Omega$}\\
u=0 &\text{on $\partial \Omega$}.
\end{cases}
\] 
\end{theorem}

Thanks to the above stability result, we can now prove the existence of solutions
by first approximating the right hand side with atomic measures, and then solving the latter problem
via a Perron-type argument, see \cite[Theorem 1.6.2]{G} for more details.

\begin{theorem}
Let $\Omega$ be a bounded open convex domain, and let $\mu$ be a nonnegative Borel measure in $\Omega$.
Then there exists an Alexandrov solution of
\begin{equation}
\label{eq:MA dirichlet}
\begin{cases}
\det D^2 u=\mu \quad&\text{in \(\Omega\)}\\
u =0 &\text{on \(\p \Omega\)}.
 \end{cases}
\end{equation}
\end{theorem}
\begin{proof}[Sketch of the proof]
Let $\mu_k=\sum_{i=1}^k \a_i \delta_{x_i}$, $\alpha_i \geq 0$, be a family of atomic measures which
converge weakly to $\mu$.
By the stability result from Theorem \ref{thm:stability} it suffices to construct a solution for $\mu_k$.

For this, we consider the family of all subsolutions \footnote{The name subsolution
is motivated by  Lemma \ref{ch2:lem:comparison}. Indeed, if $v \in S[\mu_k]$
and $u$ solves \eqref{eq:MA dirichlet} with \(\mu=\mu_k\) then $v \leq u$.}
$$
S[\mu_k]:= \{ v:\Omega \to \R\,:\, v \text{ convex},\,
 v =0 \text{ on \(\p \Omega\)},\,  \det D^2 v \geq \mu_k\}.
$$
First of all we notice that $S[\mu_k]$ is nonempty: indeed, 
it is not difficult to check that a function in this set is given by
 $$
 -A\sum_{i=1}^k C_{x_i},
 $$
 where $C _{x}$ is the ``conical'' one-homogeneous function
 which takes value $-1$ at $x$ and vanishes on $\partial \Omega$,
 and $A>0$ is a sufficiently large constant.

Then, by a variant of the argument used in the proof of Lemma \ref{ch2:lem:comparison}
one shows that
$$
v_1,v_2 \in S[\mu_k]\quad \Rightarrow \quad \max\{v_1,v_2\} \in S[\mu_k],
$$
and using Proposition \ref{uni}
one sees that  the set $S[\mu_k]$ is also closed under suprema.
Hence the function $u_k:=\sup_{v \in S[\mu_k]}v$
belongs to $S[\mu_k]$, and one wants to show that $u_k$ is actually a solution
(that is, it satisfies $\det D^2u_k=\mu_k$).

\begin{figure}
\begin{minipage}[t]{0.48\textwidth}
\includegraphics[scale=1.0]{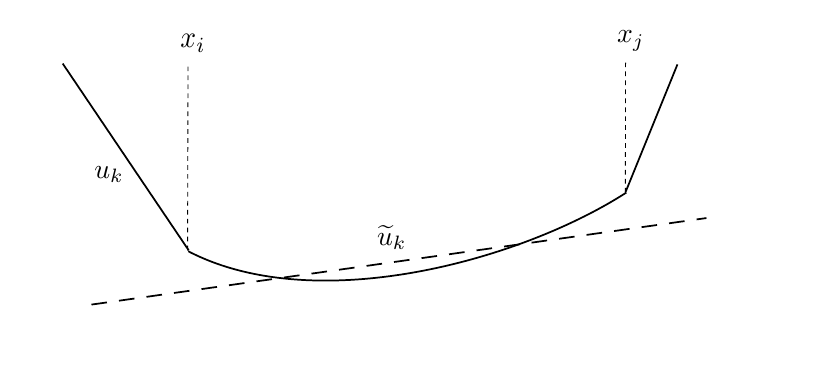}
\end{minipage}
\hfill
\begin{minipage}[t]{0.48\textwidth}
\includegraphics[scale=1.0]{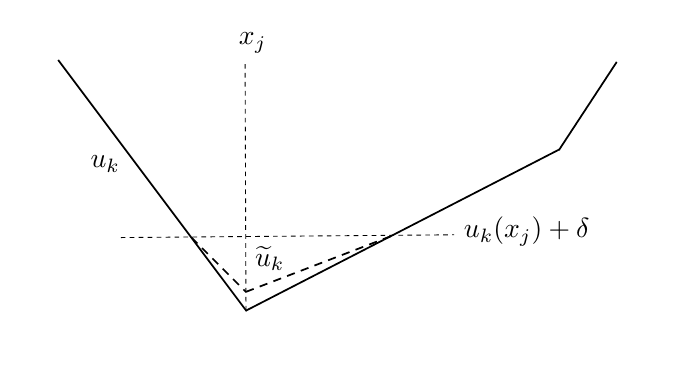}
\end{minipage}
\caption{\small{On the left, the function $\widetilde u_k$ is obtained by cutting the graph of $u_k$
with a supporting hyperplane at some point $\bar x \in \Omega\setminus \{x_1,\ldots,x_k\}$.
On the left, the function $\widetilde u_k$ is obtained by vertically dilating by a factor $(1-\delta)$
the graph of $u_k$ below the level $u_k(x_j)+\delta$.
}}
\label{MSRI}
\end{figure}

To prove this fact, one 
first shows that $\det D^2u_k$ is a measure concentrated on the set of points $\{x_1,\ldots,x_k\}$.
Indeed, if not, there would be at least
 a point $\bar x \in \Omega\setminus \{x_1,\ldots,x_k\}$ and a vector  $p \in \R^n$ such that $p \in \partial u(\bar x)\setminus \partial u(\{x_1,\ldots,x_k\})$. This means that 
 \[
 u_k(x_j)>u(\bar x)+p\cdot(x_j-\bar x) \qquad \forall \,j  \in\{1,\ldots,k\},
 \]
hence
 we can define
 $$
 \widetilde u_k(x):=\max\{u_k(x),u_k(\bar x)+p\cdot (y-\bar x)+\delta\}
 $$
 for some $\delta>0$ sufficiently small to find a larger subsolution  (see Figure \ref{MSRI}), contradiction.
 
 Then one proves that $\det D^2u_k=\mu_k$. Indeed, if this was not the case,
 we would get that $\det D^2u_k=\sum_{i=1}^k \b_i \delta_{x_i}$ with $\beta_i\geq \a_i$,
 and $\beta_j>\a_j$
 for some $j \in \{1,\ldots,k\}$. Consider $p$ in the interior of $\partial u(x_j)$
 (notice that $ \partial u(x_j)$ is a convex set of volume $\beta_j>0$, hence it has nonempty interior)
 and assume without loss of generality that $p=0$ (otherwise simply subtract $p\cdot y$ from $u_k$).
 Then we define the function
 $$
  \widetilde u_k(x):=
  \left\{
\begin{array}{ll}
u_k(x) & \text{if $u_k>u_k(x_j)+\delta$,}\\
(1-\delta)u_k(x)+\delta[u_k(x_j)+\delta] & \text{if $u_k \leq u_k(x_j)+\delta$,}
\end{array}  
 \right.
 $$
 for some $\delta>0$ sufficiently small (see Figure \ref{MSRI}) 
 and observe that this is a larger subsolution,
 again a contradiction.

Finally, the fact that $u_k=0$ on $\partial \Omega$
follows form the bound $u_k(x) \geq -C{\rm dist}(x,\partial\Omega)^{1/n}$ 
which is a consequence of  Theorem \ref{ch2:tim:alemax}.
\end{proof}

\medskip    \subsection{The continuity method and existence of smooth solutions}
\label{sect:continuity}

Existence of smooth solutions to the Monge-Amp\`ere equation dates back to the work of Pogorelov. The way  they are obtained (together with nice and useful regularity estimates)  is through the well-celebrated \emph{method of continuity} which now we briefly describe (see \cite[Chapter 17]{GT} for a more detailed exposition). Let us assume that we know how to find a smooth (convex) solution \(\bar u\) to 
\[
\begin{cases}
\det D^2 \bar u=\bar f\quad&\text{in \(\Omega\)}\\
\bar u =0 &\text{on \(\p \Omega\)}
 \end{cases}
 \]
 and that we want  to find a solution to
 \footnote{
  Here  we are considering only the case in which \(\bar f=\bar f(x)\) and \(f=f(x)\), i.e., there is no dependence on the right hand side from the ``lower order'' terms \(u\) and \(\n u\). The case \(f=f(x,u,\n u)\) is just  more tedious but the ideas/techniques are essentially the same. Note however that, in this case, one has to assume \(\partial_u f\le 0\) in order to apply the classical elliptic theory (in particular,
 the maximum principle) to the linearized operator, see for instance \cite[Chapter 17]{GT}. 
 }
 \begin{equation}\label{maloc}
\begin{cases}
\det D^2  u= f\quad&\text{in \(\Omega\)}\\
 u =0 &\text{on \(\p \Omega\)}.
 \end{cases}
 \end{equation}

 Let us  define \(f_t=(1-t)\bar f+tf\), \(t\in [0,1]\), and  consider the \(1\)-parameter family of problems
  \begin{equation}\label{MAt}
\begin{cases}
\det D^2  u_t= f_t\quad&\text{in \(\Omega\)}\\
 u_t =0 &\text{on \(\p \Omega\)}.
 \end{cases}
 \end{equation}
The method of continuity consists in showing that  the set of \(t \in [0,1]\) such that \eqref{MAt} is solvable is both open and closed, which implies the existence  of a solution  to our original problem.  More precisely, let us assume that \(f,\bar f\) are smooth and consider the set
 \[
 \mathcal C:=\{u:\overline \Omega \to \R \text{ convex functions of class \(C^{2,\alpha}(\overline \Omega)\), \(u=0\) on \(\p \Omega\)}\}.
 \]
Consider the non-linear map
 \[
 \begin{split}
 \mathcal F\colon \mathcal C\times[0,1]&\longrightarrow C^{0,\alpha}(\overline \Omega)\\
(u,t)&\mapsto  \det D^2 u-f_t.
\end{split}
 \]
We would like to show  that
\[
\mathcal T:=\{t\in [0,1]: \text{ there exists a \(u_t \in \mathcal C\) such that \(\mathcal F(u_t,t)=0\)}\},
\] 
is both open and closed inside $[0,1]$ (recall that, by assumption, $0 \in \mathcal T$).
Openness follows from the Implicit Function Theorem in Banach spaces (see \cite[Theorem 17.6]{GT}). Indeed, the Frech\`et differential of \(\mathcal F\) with respect to \(u\) is given by the linearized Monge-Amp\`ere operator: 
\begin{equation}\label{linearized}
D_u \mathcal F(u,t)[h]=\det( D^2 u )u^{ij} h_{ij}\,,\qquad h=0\text{ on \(\p \Omega\),}
\end{equation}
where we have set \(h_{ij}:=\partial_{ij} h\),  \(u^{ij}\) is the inverse of \( u_{ij}:=\partial_{ij}u\), and we are summing over repeated indices. Notice that if  \(u\) is bounded in  \(C^{2,\alpha}\) and  \(f\) is bounded from below by \(\lambda\), then the smallest eigenvalue of \(D^2 u\) is bounded uniformly away from zero and the linearized operator  becomes uniformly elliptic and with \(C^{0,\alpha}\) coefficients
(see also Remark \ref{rmk:degen}).
Therefore, classical Schauder's theory gives  the invertibility of  \(D_u \mathcal F(u,t)\) \cite[Chapter 6]{GT}.

The task is now to prove closedness of \(\mathcal T\). This is done through \emph{a priori estimates} both at the  interior  and at the boundary. As  we already noticed in Remark \ref{rmk:degen},  the Monge-Amp\`ere equation becomes uniformly elliptic on uniformly convex functions. Since \(\det D^2 u\) is bounded away from zero,   the main task is to establish an a priori bound on the \(C^2\) norm of \(u\) in \(\overline \Omega\) since this will imply that  the smallest eigenvalue of \(D^2 u\) is bounded away from zero. Once the equation becomes uniformly elliptic, Evans-Krylov Theorem \cite[Theorem 17.26 \negthickspace\negthickspace\'{}]{GT} will provide a -priori \(C^{2,\alpha}\) estimates up to the boundary, from which the closedness of \(\mathcal T\) follows by the Ascoli-Arzel\`a Theorem.
 
\begin{theorem}\label{po2}Let \(\Omega\) be a uniformly convex domain \footnote{We say that a domain is uniformly convex if  there exists a radius \(R\) such that
\[\Omega\subset B_{R}(x_0+R\nu_{x_0})
\qquad \text{for every \(x_0\in \partial \Omega\)},
\] where \(\nu_{x_0}\) is the interior normal to \(\Omega\) at \(x_0\). Note that for a smooth  domain this is equivalent to ask the second fundamental form of \(\partial \Omega\) to be (uniformly) positive definite.}
of class
\(C^3\), and let \(u\) be a solution of \eqref{maloc} with \(f\in C^2(\overline \Omega)\) and  \(\lambda \le f\le1/\lambda\). Then there exists a constant \(C\), depending only on \(\Omega\), \(\lambda\), \(\|f\|_{C^2(\overline \Omega)}\), such that
\[
\|D^2 u\|_{C^0(\overline \Omega)}\le C.
\] 
\end{theorem}
Notice that the uniform convexity of $\Omega$ is necessary to obtain regularity up to the boundary:
indeed, if $D^2u$ is uniformly bounded then (as we mentioned above)  \(u\) is uniformly convex  on \(\overline \Omega\) and  hence by the Implicit Function Theorem   $\partial \Omega= \{u=0\}$
is uniformly convex as well.

Theorem \ref{po2} together with an approximation procedure allows us to run the strategy described above to obtain the following existence result.
\begin{theorem}\label{existencesmooth}Let \(\Omega\) be a uniformly convex domain
of class $C^3$. Then for all \(f\in C^{2}(\overline \Omega)\) with \(\lambda \le f\le 1/\lambda\)
there exists a (unique) \(C^{2,\alpha}(\overline \Omega)\) solution to \eqref{maloc}.
\end{theorem}

The proof of Theorem \ref{po2} is classical. However, since the ideas involved are at the basis of many other theorems in elliptic regularity,
we give a sketch of the proof for the interested readers.

\begin{proof}[Sketch of the proof of Theorem \ref{po2}]\mbox{\\ }
We begin by noticing that, because the linearized operator in \eqref{linearized} is degenerate elliptic (in the sense that,
since we do not know yet that the eigenvalues of $D^2 u$
are bounded away from zero and infinity,
we cannot use any quantity involving its ellipticity constants), the maximum principle is essentially the only tool  at our disposal.\\

\noindent
\emph{Step 1: \(C^0\) and \(C^1\)  estimates}. \(C^0\) estimates can be obtained by a simple barrier construction.
Indeed, it suffices to use Lemma \ref{ch2:lem:comparison}
with $v(x):=\lambda^{-1/n}\bigl(|x-x_1|^2-R^2\bigr)$
(where $x_1$ and $R$ are chosen so that
\( \Omega\subset B_{R}(x_1)\))
to obtain a uniform lower bound on $u$.

To estimate the gradient we note that, by convexity,
\[
\sup_{\Omega}|\nabla u|=\sup_{\partial \Omega}|\nabla u|,
\]
so we need to estimate it only on the boundary. Since \(u=0\) on \(\partial \Omega\), any tangential derivative is zero, hence we only have to estimate the normal derivative. This can be done again with a simple barrier argument as above choosing, for any point $x_0 \in \partial \Omega$,
$$
v_\pm(x):=\lambda^{\mp 1/n}\bigl(|x-x_\pm|^2-R_\pm^2\bigr)
$$
where 
$$
x_\pm:=x_0+R_\pm\nu_{x_0}
$$
and $0<R_-<R_+<\infty$ are chosen so that
$$
B_{R_-}(x_-)\subset \Omega\subset B_{R_+}(x_+).
$$
In this way we get $v_+ \leq u \leq v_-$, therefore 
\begin{equation}\label{normal}
-C\le \partial_\nu u(x_0)\le -\frac{1}{C}.
\end{equation}\\

\noindent
\emph{Step 2: \(C^2\) estimates}. This is the most delicate step.  Given a unit vector \(e\) we differentiate  the equation
\begin{equation}\label{log}
\log \det D^2u=\log f
\end{equation}
once and two times in the direction of $e$ to get, respectively,
\begin{equation}\label{lin1}
L(u_e):=u^{ij}  (u_e)_{ij}=(\log f)_{e}
\end{equation}
and
\begin{equation}\label{lin2}
u^{ij} (u_{ee})_{ij}-u^{il}u^{kj}(u_e)_{ij}( u_e)_{lk}=(\log f)_{ee}
\end{equation}
(recall that $u^{ij}$ denotes the inverse of $u_{ij}$ and that lower indices 
denotes partial derivatives).
 By the convexity of \(u\), \(u^{il}u^{kj}(u_e)_{ij}( u_e)_{lk} \ge0\), hence
\begin{equation*}
L(u_{ee})\ge (\log f)_{ee}\ge -C,
\end{equation*}
for some constant \(C\) depending only on \(f\). Since \(L(u)=u^{ij} u_{ij} =n\), we see that
\[
L(u_{ee}+Mu)\ge 0
\]
for a suitable large constant \(M\) depending on \(f\). Hence, by the maximum principle,
\[
\sup_{\Omega} (u_{ee}+Mu)\le \sup_{\partial \Omega} (u_{ee}+Mu).
\]
Since \(u\) is bounded, to get an estimate  \(D^2u\) we only have to estimate it on the boundary.  Let us assume that \(0\in \partial \Omega\) and that locally 
\begin{equation}\label{bordo}
\partial \Omega=\Big\{(x_1,\dots ,x_n)\,:\  x_n=\sum_{\alpha=1}^{n-1}\frac{\kappa_{\alpha}}{2} x_{\alpha}^2+O(|x|^3)\Big\}
\end{equation}
for some constants $\kappa_\a >0$.
Notice that, by the smoothness and uniform convexity of $\Omega$, we have
\(1/C \le \kappa_\alpha\le C\).
In this way
\[
 u_{\alpha\alpha}(0)=-\kappa_{\alpha}u_{n}(0),\qquad u_{\alpha\beta}(0)=0 \quad\forall\, \alpha\ne\beta\in\{1,\dots,n-1\}.
\]
Thanks to \eqref{normal} this gives 
\begin{equation}\label{tangential}
 \Id_{n-1}/C\le \bigl( u_{\alpha\beta}(0) \bigr)_{\alpha,\beta\in\{1,\dots,n-1\}}\le C\,\Id_{n-1}.  
 \end{equation}
 Noticing that  \[
 f=\det D^2 u=M^{nn}(D^2u) u_{nn}+\sum_{\alpha=1}^{n-1} M^{\alpha n}(D^2u)u_{\alpha n} 
 \]
 with \(M^{ij}(D^2 u)\) the cofactor of \(u_{ij} \),
 this identity and \eqref{tangential} will give an upper bound on \(u_{nn}(0)\) once one has an
 upper bound on the mixed derivative  \(u_{\alpha n}(0)\) for $\a \in \{1,\ldots,n-1\}$. Hence, to conclude, we only have to  provide an upper bound on \(u_{\alpha n}(0)\).
 
 For this, let us  consider the ``rotational'' derivative operator
 \[
 R_{\alpha n}=x_\alpha \partial_{n}-x_n\partial_\alpha \qquad \alpha \in \{1,\dots,n-1\}.
 \]
 By the invariance of the determinant with respect to rotations, differentiating \eqref{log}
 we get
 \[
 L(R_{\alpha n} u)=u^{ij}(R_{\alpha n} u)_{ij}= R_{\alpha n} (\log f),
 \]
 hence, multiplying the above equation by $\k_\a$ 
 and using \eqref{lin1}, we get (recall that $\k_\a$ is a constant)
  \begin{equation}\label{derivozzo}
 \Big|L\big((1- \k_{\alpha} x_n) u_\alpha+\k_{\alpha}x_\alpha u_n\big)\Big|\le C.
 \end{equation}
 Since \(u=0\) on \(\partial \Omega\), thanks to \eqref{bordo},  the uniform convexity of \(\Omega\),
 and the bound on \(|\nabla u|\), one easily computes that
 \begin{equation}\label{dato}
 \big|(1- \k_{\alpha} x_n) u_\alpha+\k_{\alpha}x_\alpha u_n\big|\le -A|x|^2+Bx_n\qquad \text{on \(\partial \Omega\)}
 \end{equation}
 for a suitable choice of constants \(B\gg A\gg1\) depending only on \(\Omega\). Moreover,
 \begin{equation}\label{sub}
 L\big( -A|x|^2+Bx_n\big)=-A\sum_{i} u^{ii}\le -\frac{nA}{(\det D^2 u)^{1/n}}\le -\frac{nA}{\lambda^{1/n}},
 \end{equation}
 where we have used the arithmetic-geometric mean inequality. Hence,
 choosing \(A\) large enough, \eqref{derivozzo}, \eqref{dato}, and \eqref{sub}, together with the comparison principle, imply
  \[
 \big|(1- \k_{\alpha} x_n) u_\alpha+\k_{\alpha}x_\alpha u_n\big|\le -A|x|^2+Bx_n\qquad \text {in \(\Omega\)}.
 \]
 Dividing by \(x_n\) and letting \(x_n\to 0\) we finally get 
   \begin{equation}\label{mixed}
  | u_{\alpha n}(0)|\le C
  \end{equation}
 for a constant \(C\) depending only on \(\Omega\) and \(f\), as desired.
 \end{proof}

\medskip    \subsection{Interior estimates and regularity of weak solutions.}

In the previous sections we have shown existence and uniqueness of weak solutions in general convex domains,
and the existence of smooth solutions in smooth uniformly convex domains.
To obtain smoothness of weak solutions we need the following interior estimate, due to Pogorelov.

\begin{theorem}[Pogorelov's interior estimate]\label{po1} Let \(u \in C^4(\Omega)\)  be a solution of \eqref{maloc} with
\(f\in C^2(\Omega)\)  and \(\lambda \le f\le1/\lambda\). Then there exist a constant \(C\), depending only on  \(\lambda\) and \(\|f\|_{C^2}\), such that
\begin{equation}\label{posticazzi}
|u(x)| u_{11}(x)e^{\frac{(u_1(x))^2}{2}}\le C\Bigl\|e^{u_1^2/2}(1+|u_1|+|u|) \Bigr\|_{L^\infty(\Omega)}\qquad \text{for all \(x\in \Omega\).}
\end{equation}
\end{theorem}

\begin{proof}[Sketch of the proof] Observing that (by convexity) $u \leq 0$ in $\Omega$,
let us define 
\[
w=(-u)u_{11}e^{\frac{(u_1)^2}{2}}\qquad x\in \Omega.
\]
Let \(x_0\) a maximum point of \(w\) in \(\overline \Omega\) and notice that \(x_0\in \Omega\) thanks to the boundary condition \(u=0\). 

First we make a change of coordinate \(x'=Ax\) with \(\det A=1\)  which leaves the direction \(x_1\) invariant and such that \(u_{ij}\) is diagonal at \(x_0\) (see for instance \cite[Chapter 4]{G}). 
Then we compute
\[
\begin{aligned}
(\log w)_{i}&=\frac{u_i}{u}+\frac{u_{11i}}{u_{11}}+u_{1}u_{1i}, \\
(\log w)_{ij}&=\frac{u_{ij}}{u}-\frac{u_i u_j}{u^2}+\frac{u_{11ij}}{u_{11}}-\frac{u_{11i}u_{11j}}{(u_{11})^2}+u_{1}u_{1ij}+u_{1i}u_{1j}.
\end{aligned}
\]
Since \(x_0\) is a maximum point for \(\log w\),  we have 
\begin{equation}\label{grad}
0=(\log w)_{i}=\frac{u_i}{u}+\frac{u_{11i}}{u_{11}}+u_{1}u_{1i},
\end{equation}
and
\begin{equation*}\label{contazzo}
\begin{aligned}
0\ge u^{ij} (\log w)_{ij}&=\frac{u^{ij}u_{ij}}{u}-\frac{u^{ij}u_i u_j}{u^2}+\frac{u^{ij}u_{11ij}}{u_{11}}-\frac{u^{ij}u_{11i}u_{11j}}{(u_{11})^2}+u_{1}u^{ij}u_{1ij}+u^{ij}u_{1i}u_{1j}\\
&=\frac{L(u)}{u}-\frac{u^{ij}u_i u_j}{u^2}+\frac{L(u_{11})}{u_{11}}-\frac{u^{ij}u_{11i}u_{11j}}{(u_{11})^2}+u_1 L(u_1)+u^{ij}u_{1i}u_{1j},
\end{aligned}
\end{equation*}
where again \(L(h)=u^{ij}h_{ij}\) and all functions are evaluated at \(x_0\). Using \eqref{lin1} and \eqref{lin2} with \(e=e_1\) and  that \(L(u)=n\) we get
\begin{equation}\label{contazzo2}
0\ge \frac{n}{u}-\frac{u^{ij}u_i u_j}{u^2}+\frac{(\log f)_{11}}{u_{11}}+\frac{u^{il}u^{kj}u_{1ij} u_{1kl}}{u_{11}}-\frac{u^{ij}u_{11i}u_{11j}}{(u_{11})^2}+u_1 (\log f)_{1}+u^{ij}u_{1i}u_{1j}.
\end{equation}
Now, recalling that \(u_{ij}\) and \(u^{ij}\) are  diagonal at \(x_0\), thanks to \eqref{grad} we obtain 
\[
\begin{split}
\frac{u^{il}u^{kj}u_{1ij} u_{1kl}}{u_{11}}&-\frac{u^{ij}u_{11i}u_{11j}}{(u_{11})^2}-\frac{u^{ij}u_i u_j}{u^2}\\
&=\frac{u^{il}u^{kj}u_{1ij} u_{1kl}}{u_{11}}-\frac{u^{ij}u_{11i}u_{11j}}{(u_{11})^2}-u^{ij}\Big(\frac{u_{11i}}{u_{11}}+u_{1}u_{1i}\Big)\Big(\frac{u_{11j}}{u_{11}}+u_{1}u_{1j}\Big)\\
&\ge- u^{11}\Big(\frac{u_{111}}{u_{11}}+u_{1}u_{11}\Big)^2=-\frac{(u_1)^2}{u^2 u_{11}}\, .
\end{split}
\]
Plugging the above equation in \eqref{contazzo2} and taking again into account that \(u_{ij}\) is diagonal, we get
\[
0\ge \frac{n}{u}+\frac{(\log f)_{11}}{u_{11}}+u_1 (\log f)_{1}-\frac{(u_1)^2}{u^2 u_{11}}+u_{11}.
\]
Multiplying by \(u^2u_{11}e^{(u_1)^2}\) and recalling the definition of \(w\) we finally obtain
\[
0\ge w^2-Cwe^{\frac{(u_1)^2}{2}}(1+|u_1|u)-Ce^{(u_1)^2}(u^2+(u_1)^2)
\]
for a suitable constant \(C\)  depending only on \(f\),
proving \eqref{posticazzi}.
\end{proof}

Combining the theorem above with the stability of weak solutions and the existence of smooth solutions,
we obtain the following result:

\begin{theorem}\label{thm:pogsmooth} Let $u:\Omega  \to \R$ be a convex  Alexandrov solution of \eqref{maale}
with \(f\in C^2(\Omega)\) and \(\lambda \le f\le1/\lambda\). 
Assume that $u$ is strictly convex inside $\Omega'\subset \Omega$.
Then $u \in C^2(\Omega')$.
\end{theorem}

\begin{proof}[Sketch of the proof]
Fix $x_0 \in \Omega'$, $p \in \partial u(x_0)$, and consider the \emph{section} of \(u\) at height \(t\) defined as
\begin{equation}\label{sec}
S(x,p,t):=\big\{y\in \Omega\,:\, u(y)\le u(x)+p\cdot(y-x)+t\big\}.
\end{equation}
Since \(u\) is strictly convex we can choose $t>0$ small enough so that
$S(x_0,p,t)\Subset \Omega'$.
Then we consider $S_\e$ a sequence of smooth uniformly convex sets converging to $S(x_0,p,t)$
and apply Theorem \ref{existencesmooth} to find a function $v_\epsilon \in C^{2,\alpha}(S_\e)$
solving
$$\begin{cases}
\det D^2  v_\epsilon= f\ast\rho_\epsilon \quad&\text{in \(S_\e\)}\\
 v_\epsilon =0 &\text{on \(\p S_\e\)}.
 \end{cases}
$$
By Schauder's theory $v_\epsilon$ are of class $C^\infty$
inside $S_\e$, so we can apply Theorem \ref{po1}
to deduce that 
$$
|D^2v_\e| \leq C \qquad \text{in $S(x_0,p,t/2)$}
$$
for $\e$ sufficiently small.
Since $S_\e \to S(x_0,p,t)$ and $u(x)=u(x_0)+p\cdot x+t$ on $\partial S(x_0,p,t)$,
by uniqueness of weak solutions we deduce that $v_\e+u(x_0)+p\cdot x+t \to u$ uniformly as
 $\e \to 0$, hence 
 $|D^2u| \leq C$ in $S(x_0,p,t/2)$. 
This makes the equation uniformly elliptic, so $u \in C^{2}(S(x_0,p,t/4))$.
By the arbitrariness of $x_0$ we obtain that $u \in C^2(\Omega')$, as desired.
\end{proof}

\medskip    \subsection{Further regularity results for weak solutions}
\label{sect:Caff Alex}

In the 90's
Caffarelli developed a regularity theory for Alexandrov solutions, showing that strictly convex solutions of \eqref{eq:MAclassical}
are locally $C^{1,\alpha}$ provided $0<\lambda \leq f\leq 1/\lambda$
for some $\lambda \in \mathbb{R}$ \cite{Caf1,Caf2,Caf3}. 

\begin{theorem}[Caffarelli]\label{thm:calfa}Let \(u:\Omega \to \R\) be a {strictly convex} solution of \eqref{maale} with \(\lambda \le f\le 1/\lambda\). Then \(u\in C_{\rm loc}^{1,\alpha}(\Omega)\) for some universal \(\alpha\). More precisely for every \(\Omega'\Subset \Omega\) there exists a constant  \(C\), depending on \(\lambda\), \(\Omega'\), and the modulus of strict convexity of \(u\), such that
\[
\sup_{\substack{ x,y \in \Omega'\\ x\ne y}}\frac{|\nabla u(x)-\nabla u(y)|}{|x-y|^{\alpha}}\le C.
\]
\end{theorem}

To explain the ideas behind the proof of the above theorem let us point out the following simple properties of solutions to the Monge-Amp\`ere equation (which are just another manifestation of its degenerate ellipticity):
If \(A\) is a linear transformation with \(\det A=1\) and \(u\) is a solution of the Monge-Amp\`ere equation
with right hand side $f$, then \(u\circ A\) is a solution to the Monge-Amp\`ere equation with right hand side $f \circ A$. This affine invariance creates serious obstructions to obtain a local regularity theory.
Indeed, for instance, the  functions  
 \[
 u_\e(x_1,x_2)=\frac{ \e x_1^{2}}{2}+\frac{x_2^{2}}{2\e}-1
 \]
 are solutions to \(\det D^2 u_\e=1\) on \(\{u_\e \le 0\}\). Thus, unless the level set $\{u_\e=0\}$ is sufficiently ``round'', there is no hope to obtain a priori estimates on \(u\). The intuition of Caffarelli was to use the so-called John's Lemma \cite{John}:
 
\begin{lemma}
\label{lem:john} Let \(\mathcal K \subset \R^n\) be a bounded convex set with non-empty interior. Then there exists a unique ellipsoid \(E\) of maximal volume contained in \(\mathcal K\). Moreover this ellipsoid satisfies
\begin{equation}\label{ap:j}
E\subset \mathcal K\subset nE,
\end{equation}
where $nE$ denotes the dilation of $E$ by a factor $n$ with respect to its center. 
\end{lemma}
In particular, if we define a convex set \(\Omega\)  to be \emph{normalized} if 
\[
B_1\subset \Omega \subset nB_{1},
\]
then Lemma \ref{lem:john} says that, for every bounded open convex set \(\Omega\), there is an affine transformation \(A\) such that \(A(\Omega)\) is normalized.  Note that if \(u\) solves
\begin{equation*}
\lambda\le \det D^2 u \le 1/\lambda\quad \text{in \(\Omega\)},\qquad u=0 \quad \text{on \(\partial \Omega\)}
\end{equation*} 
 and \(A\) normalizes \(\Omega\), then \(v:=(\det A)^{2/n}\, u\circ A^{-1}\) solves 
\begin{equation}
\label{eq:v normalized}
\lambda \le  \det D^2 v \le 1/\lambda\quad \text{in \(A(\Omega)\)},\qquad u=0 \quad \text{on \(\partial \bigl(A(\Omega)\bigr)\)}.
\end{equation}
Thanks to the above discussion and a suitable localization argument based on the strict convexity of \(u\) as in the proof of Theorem \ref{thm:pogsmooth}, Theorem \ref{thm:calfa} is a consequence of the following:

\begin{theorem}\label{calpfanorm}Let  \(\Omega\) be a normalized convex set and \(u\) be a solution of 
\eqref{maloc} with \(\lambda\le f\le 1/\lambda\). Then there exist positive constants \(\alpha=\alpha(n,\lambda)\) and \(C=C(n,\lambda)\) such that
\[
\|u\|_{C^{1,\alpha}(B_{1/2})}\le C.
\] 
\end{theorem}

In the proof of the above theorem, a key step consists in showing that solutions of \eqref{maloc} on normalized domains have a \emph{universal} modulus of strict convexity. A fundamental ingredient to prove this fact is the  following important result of Caffarelli \cite{Caf1}:
\begin{proposition}\label{ch2:thm:contact}Let \(u\) be a solution of
\[
\lambda \le \det D^2 u\le 1/\lambda
\]
inside a convex set \(\Omega\) and let \(\ell:\R^n \to \R\) be a linear function  supporting  \(u\) at some point \(\bar x\in \Omega\). If  the convex set
\[
W:=\{x\in \Omega\,:\, u(x)=\ell(x)\}
\]
contains more than one point, then it cannot have extremal points in \(\Omega\).
\end{proposition}
This statement says that, if a solution coincides with one of its supporting plane on more than one point (that is, it is not strictly convex) then the contact set has to cross the domain.
Hence, if the boundary conditions are such that $u$ cannot coincide with an affine function along a segment crossing $\Omega$, one deduce that $u$ must be strictly convex.
In particular by noticing that, when $\Omega$ is normalized, 
the class of solutions is compact with respect to the uniform convergence (this is a consequence of Proposition \ref{uni} and the fact that the family of normalized convex domains is a compact), a simple contradiction argument based on compactness shows the modulus of strict convexity must be universal. From this one deduces the following estimate:

\begin{lemma}
\label{lemma hk}
Let $\Omega$ be a normalized convex domain and let $v$ be a solution of \eqref{maloc} with \(\lambda\le f\le 1/\lambda\).
Let $x_0$ be a minimum point for $u$ in $\Omega$, and for any $\beta \in (0,1]$ denote by
\(C_{\beta}\subset \R^{n+1}\) the cone with vertex \((x_0,v(x_0))\) and base \(\{v=(1-\beta) \min v\}\times \{(1-\beta) \min v\}\). Then, if $h_\beta$ is the function
whose graph is given by $C_\beta$  (see Figure \ref{fig:c1alpha}), 
 there exists a universal constant \(\delta_{0} >0\) such that
\begin{equation}
\label{eq:h}
h_{1/2}\le (1-\delta_0)h_{1}. 
\end{equation}
\end{lemma}

Now the proof of Theorem \ref{calpfanorm} goes as follows: 
\begin{figure}
\centering
\includegraphics[scale=0.55]{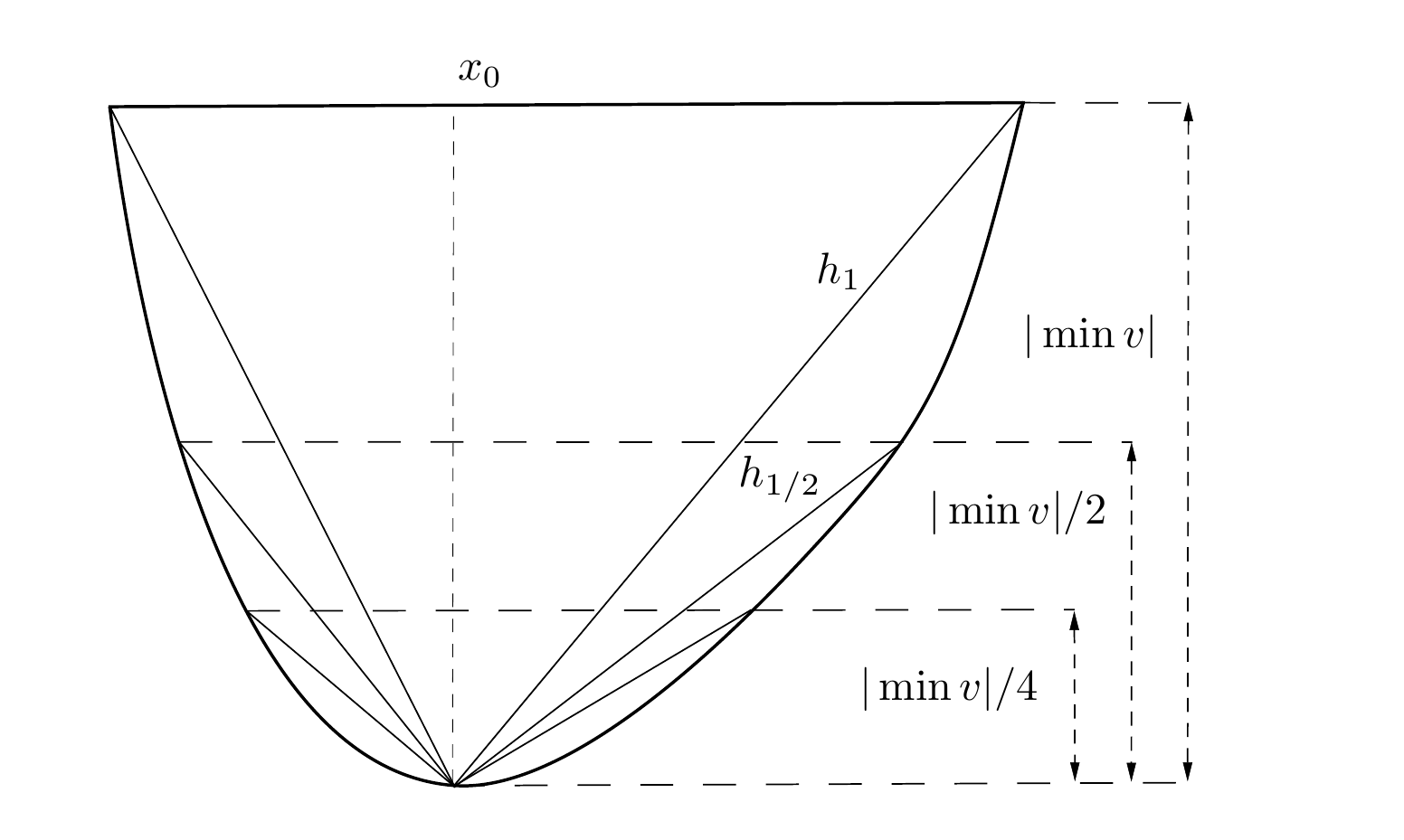}
\caption{The function \(v\) looks flatter and flatter near its minimum \(x_0\).}
\label{fig:c1alpha}
\end{figure}
for any $k \in \N$ we can consider the convex set $\Omega_k:=\{ u \leq (1-2^{-k})\min u\}$.
Then, if we renormalize $\Omega_k$ through an affine map $A_k$,  we can apply Lemma \ref{lemma hk} to the function 
$v=(\det A_k)^{2/n}u\circ A_k$ (see \eqref{eq:v normalized}) and  transfer the information
back to \(u\) to deduce that 
$h_{2^{-(k+1)}}  \le (1-\delta_0)h_{2^{-k}}$.
Therefore, by iterating this estimate we obtain 
$$
h_{2^{-k}} \leq (1-\delta_0)^kh_1\qquad \forall\,k \in \N.
$$
From this fact it follows that \(v\) is \(C^{1,\alpha}\) at the minimum,
in the sense that
$$
u(y)-u(x_0)\leq C|y-x_0|^{1+\alpha},
$$
see Figure \ref{fig:c1alpha}. 

Now, given any point $x \in \Omega'\Subset\Omega$ and $p \in \p u(x)$,
we can repeat the same argument with the function $u(y)-p\cdot (y-x)$ in place of $u$,
replacing $\Omega$ with the  section $S(x,p,t)$ (see \eqref{sec}) for some $t$ small but fixed
chosen such that 
$S(x,p,t)\Subset \Omega$.\footnote{
To be more precise, one needs to apply an affine transformation so that the section $S(x,p,t)$
becomes a normalized convex set.}
Then the above estimate gives
$$
u(y)-u(x) - p\cdot (y-x)\leq C|y-x|^{1+\alpha} \qquad \forall\,p \in \partial u(x).
$$
By the arbitrariness of $x \in \Omega'$, it is well known that the above estimate implies the desired 
$C^{1,\alpha}_{\rm loc}$ regularity of $u$ (see for instance \cite[Lemma 3.1]{DFconvex}).\\

Let us notice that a direct proof of Theorem \ref{calpfanorm} (which avoids any compactness argument) has been given by Forzani and Maldonado \cite{ForMa}, allowing one to compute the explicit dependence of \(\alpha\) on \(\lambda\).

It is important to point out that strict convexity is not just a technical assumption but it is necessary to obtain regularity. Indeed, as we already mentioned at the beginning of  Section \ref{sect:classical MA}, there are Alexandrov solutions to the Monge-Amp\`ere equation with smooth right-hand side which are not $C^2$.
For instance, the function
\begin{equation}
\label{eq:example pogorelov}
u(x_1,x'):=|x'|^{2-2/n} (1+x_1^2),\qquad n\ge 3,
\end{equation}
is $C^{1,1-2/n}$ and solves $\det D^2u=c_n(1+x_1^2)^{n-2}(1-x_1^2)>0$ inside $B_{1/2}$.
Furthermore, having only the bound \(\lambda \le \det D^2 u\le 1/\lambda\) is not even enough for $C^1$
regularity:
the function
\[
u(x_1,x'):=|x'|+|x'|^{n/2}(1+x_1^2),\qquad n\ge 3,
\]
is merely Lipschitz and solves \(\lambda \le \det D^2 u\le 1/\lambda\)
in a small convex neighborhood of the origin.

Alexandrov showed in \cite{AL2} that, in contrast with the above counterexamples, in two dimension every solution of  \(\det D^2 u\ge \lambda\) is strictly convex (see also \cite{Caf-CPDE}). Recently, in \cite{Mooney}, Mooney established that, in every dimension, the set of points where an Alexandrov solution of \(\det D^2 u\ge \lambda\) coincides with one of its supporting plane has vanishing \(\mathcal H^{n-1}\) measure. 

In the case when $f$ is  H\"older continuous, Caffarelli proved that \(u\) is locally \(C^{2,\alpha}\)
\cite{Cafw2p},
improving Theorem \ref{thm:pogsmooth}:

\begin{theorem}Let  \(\Omega\) be a normalized convex set and \(u\) be an Alexandrov solution of
\eqref{maloc} with \(\lambda\le f\le 1/\lambda\) and  \(f\in C^{0,\alpha}(\Omega)\).
Then $\|u\|_{C^{2,\alpha}(B_{1/2})}\le C$ for some constant $C$ depending only on $n$, $\lambda$,
and $\|f\|_{C^{0,\alpha}(B_{1})}$.
\end{theorem}

The proof of the above theorem is based on showing that, under the assumption that \(f\) is almost a constant (say very close to \(1\)), \(u\) is very close to the solution of \eqref{maloc} with right hand side \(1\). Since this latter function has interior a priori estimates (by Theorem \ref{po1}),
an approximation/interpolation argument permits to show that the \(C^2\) norm of \(u\) remains bounded. With this line of reasoning one  can also prove the following theorem of Caffarelli \cite{Cafw2p}:

\begin{theorem}Let  \(\Omega\) be a normalized convex set and \(u\) be a solution of \eqref{maloc}.
Then, for every \(p>1\), there exist positive constants \(\delta(p)\) and \(C=C(p)\) such that if \(\|f-1\|_\infty \le \delta(p)\), then \(\|u\|_{W^{2,p}(B_{1/2})}\le C.\)
\end{theorem}

We notice that, by localizing the above result on small sections one deduces that
if $u$ is a strictly convex solutions of \eqref{eq:MA} with $f$ continuous then $u \in W^{2,p}_{\rm loc}$
for all $p<\infty$ (see \cite{Cafw2p} or \cite[Chapter 6]{G}).

Wang \cite{wang} showed that for any $p>1$ there exists a function $f$
satisfying $0<\lambda(p) \leq f\leq 1/\lambda(p)$ such that $u \not \in W^{2,p}_{\rm loc}$.
This counterexample shows that the results of Caffarelli are more or less optimal.
However, an important question which remained open for a long time was whether solutions of \eqref{eq:MA}
with $0<\lambda \leq f\leq 1/\lambda$
could be at least $W^{2,1}_{\rm loc}$,
or even $W^{2,1+\varepsilon}_{\rm loc}$ for some $\varepsilon=\varepsilon(n,\lambda)>0$. The question of \(W_{\rm loc}^{2,1}\) regularity has been recently solved by the authors in \cite{DF1}. Following the ideas introduced there, later  in  \cite{DFS, Shm} the result has been refined to \(u\in W_{\rm loc}^{2,1+\e}\) for some \(\e>0\).

\begin{theorem}\label{w21eps}Let  \(\Omega\) be a normalized convex set and \(u\) be be an Alexandrov solution of
\eqref{maloc}
 with \(\lambda\le f\le 1/\lambda\).
Then there exist positive constants \(\e=\e(n,\lambda)\) and \(C=C(n,\lambda)\) such that
$\|u\|_{W^{2,1+\e}(B_{1/2})}\le C.$ 
\end{theorem}

Note that all the previous results hold for strictly convex solutions of \eqref{eq:MA}.
Indeed, it suffices to pick a section $S(x,p,t)\Subset \Omega$,
choose an affine transformation $A$ so that $A\bigl(S(x,p,t) \bigr)$
is normalized (see \eqref{eq:v normalized}),
and then apply the results above with $A\bigl(S(x,p,t) \bigr)$ in place of $\Omega$.\\

We now briefly sketch the  ideas behind the proof of the \(W^{2,1}\) regularity in \cite{DF1}.
To prove that the distributional Hessian \(D_\DD^2 u\) of an
Alexandrov solution is a \(L^1\) function, it is enough to prove an a priori
equi-integrability estimate on smooth solutions.\footnote{Note that  it is pretty simple to show that for smooth convex functions \(\|D^2 u\|_{L^1}\le C\osc u\). However, since
an equi-bounded sequence of \(L^1\) functions can converge to a singular measure, this  is not enough to prevent  concentration of the Hessians in the limit,
and this is why we need an equi-integrability estimate on $D^2u$
which is independent of $u$.} To do this, a key observation is that any domain  \(\Omega\) endowed with the Lebesgue measure and the family of ``balls'' given
by the  sections \(\{S(x,p,t)\}_{x\in \Omega,\, t\in \R}\) of solutions  of \eqref{maloc} as defined in \eqref{sec} is a space homogenous type  in the sense of Coifman
and Weiss, see \cite{CG, GH, AFT}. In particular Stein's Theorem implies that if 
\[
\mathcal M(D^2 u)(x):=\sup_{t}\mean{S(x,p,t)} |D^2 u| \in L^1,
\]
then \(|D^2 u|\in L\log L\), that is $\int |D^2 u| \log(2+|D^2u|) \leq C$. Hence one wants to prove that
$\mathcal M(D^2u) \in L^1$, and the key estimate in \cite{DF1} consists in showing that
\[
\|\mathcal M(D^2 u)\|_{L^1}\le C \|D^2 u\|_{L^1},
\]
for some constant \(C=C(n,\lambda)\), which proves the result.\\

 A more careful applications of these ideas then gives a priori \(W^{2,1+\e}\) estimates. 
As shown for instance in \cite{F-proc}, the approach in \cite{DFS}
can be also used to give a very direct and short proof of the $W^{2,p}$ estimates of Caffarelli in \cite{Cafw2p}. 

Using Theorem \ref{w21eps}, the authors could show in \cite{DFstab} the following stability property of solutions  of \eqref{maale}  in the \emph{strong} Sobolev topology. Due the highly non-linear character of the Monge-Amp\`ere equation, the result is nontrivial.

\begin{theorem}\label{conv} 
Let $\Omega_k\subset \R^n$ be a  family of bounded convex domains, and  let $u_k:\Omega_k \to \R$ be convex Alexandrov solutions of
\begin{equation*}
\begin{cases}
\det D^2 u_k=f_k \quad &\text{in $\Omega_k$}\\
u_k=0 &\text{on $\partial \Omega_k$}
\end{cases}
\end{equation*}
with  $0<\lambda \le f_k \le 1/\lambda$. Assume that $\Omega_k$ converge to some bounded convex domain $\Omega$
in the Hausdorff distance, and $f_k \mathbf 1_{\Omega_k}$ converge to $f$ in $L^1_{\rm loc}(\Omega)$. Then, if $u$
denotes the unique Alexandrov solution of 
\[
\begin{cases}
\det D^2 u=f \quad &\text{in $\Omega$}\\
u=0 &\text{on $\partial \Omega$},
\end{cases}
\] 
for any $\Omega'\Subset  \Omega$ we have
\begin{equation*}\label{L1econv}
\|u_k-u\|_{W^{2,1}(\Omega')}\to 0 \qquad \text{as $k \to \infty$}.
\end{equation*}
\end{theorem}

Let us conclude this section mentioning that Savin recently introduced new techniques to obtain
global versions of all the above regularity results 
under suitable regularity assumptions on the boundary \cite{Savin localization,Savin boundary,SavinW2p}.

\medskip    \subsection{An application: global existence for the semigeostrophic equations}\label{semigeo}

Let us conclude this discussion on the regularity of weak solutions
by showing an application of Theorem \ref{w21eps} to prove the existence of distributional solutions for  the semigeostrophic system.

The semigeostrophic equations are a simple model used
in meteorology to describe large scale atmospheric flows.
As explained for instance in \cite[Section 2.2]{bebr}
(see also \cite{cu} for a more complete exposition),
these equations can be derived from the 3-d 
Euler equations, with Boussinesq and hydrostatic approximations, subject to a strong
Coriolis force.
Since for large scale atmospheric flows the Coriolis force dominates
the advection term, the flow is mostly bi-dimensional. For this reason,
the study of the semigeostrophic equations in 2-d or 3-d is pretty similar,
and in order to simplify our presentation we focus here on the 2-dimentional periodic case.

The semigeostrophic system can be written as
\begin{equation}\label{eqn:SGsystem2}
\begin{cases}
\partial_t \nabla p_t + (\uu_t\cdot \nabla) \nabla p_t  +\nabla^\perp p_t +\uu_t = 0\\
\nabla \cdot \uu_t = 0 \\
p_0= \bar p
\end{cases}
\end{equation}
where $\uu_t:\R^2 \to \R^2$ and $p_t:\R^2 \to \R$ are periodic functions
corresponding respectively to the velocity and the pressure, and \(\nabla^\perp p_t \) is the \(\pi/2\) counterclockwise rotation of \(\nabla p\).

As shown in \cite{cu},
energetic considerations show that it
is natural to assume that $p_t$ is ($-1$)-convex, i.e., the function
$P_t(x):=p_t(x)+|x|^2/2$ is convex on $\R^2$. If we denote with
$\mathcal L_{\T^2}$ the Lebesgue measure on the 2-dimensional torus, then
formally \footnote{See \eqref{pushforward} for the definition of $(\nabla P_t)_\sharp \mathcal L_{\T^2}$.} $\rho_t:=(\nabla P_t)_\sharp \mathcal L_{\T^2}$ \
satisfies the following \emph{dual problem} (see for instance \cite[Appendix]{ACDF1}):

\begin{equation}\label{eqn:dualsystem}
\begin{cases}
\partial_t \rho_t +\nabla \cdot (\boldsymbol {\mathcal U}_t \rho_t) = 0 \\
\boldsymbol {\mathcal U}_t(x) = \bigl(x-\nabla P_t^*(x)\bigr)^\perp\\
\rho_t= (\nabla P_t)_{\sharp} \mathcal L_{\T^2}\\
P_0(x)= \bar p(x)+|x|^2/2,
\end{cases}
\end{equation}
where $P^*_t$ is the convex conjugate of $P_t$, namely
\[
P_t^*(y):=\sup_{x\in\R^2} \bigl\{y \cdot x-P_t(x)\bigr\}.
\]
The dual problem \eqref{eqn:dualsystem} is nowadays pretty well
understood. In particular, Benamou and Brenier proved in
\cite{bebr} existence of weak solutions to \eqref{eqn:dualsystem}.
On the contrary, much less is
known about the original system \eqref{eqn:SGsystem2}. Formally,
given a solution $(\rho_t,P_t)$ of \eqref{eqn:dualsystem}, the pair $(p_t,\uu_t)$ given
by
\begin{equation}\label{eqn:velocity}
\begin{cases}
p_t(x):=P_t(x)-|x|^2/2& \cr \uu_t(x):= \partial_t\nabla P_t^* (\nabla
P_t(x)) + D^2 P_t^* ( \nabla P_t(x))\,\bigl(\nabla P_t(x) - x\bigr)^\perp&
\end{cases}
\end{equation}
solves \eqref{eqn:SGsystem2}. 

Being $P^*_t$ just a convex
function, \emph{a priori} $D^2 P_t^*$ is a matrix-valued measure,
thus it is not clear the meaning to
give to the previous formula.
However, since $\rho_t$ solves a continuity equation
with a divergence free vector field (notice that $\boldsymbol {\mathcal U}_t $
is the rotated gradient of the function $|x|^2/2-P_t^*(x)$, see \eqref{eqn:dualsystem}),
the only available bound on $\rho_t$ is
\begin{equation}
\label{eq:bounds rhot}
0<\lambda \leq \rho_t\leq 1/\lambda \qquad \forall\,t >0
\end{equation}
provided this bound holds at $t=0$.

In addition, the relation $\rho_t= (\nabla P_t)_\sharp\mathcal L_{\T^2}$
implies that $ (\nabla P_t^*)_\sharp\rho_t=\mathcal L_{\T^2}$
(since $\nabla P_t^*$ is the inverse of $\nabla P_t$), from which it follows  that $P_t^*$
solves in the Alexandrov sense the Monge-Amp\`ere equation
$$
\det(D^2P_t^*)=\rho_t
$$
(see Section \ref{sect:Alex sol} for the definition of Alexandrov solution and Theorem \ref{thm:regularity transport} below).
Hence,  by Theorem \ref{w21eps}
$D^2P_t^* \in L^{1+\e}$, which allows us to give a meaning to the velocity field $\uu_t$ defined in \eqref{eqn:velocity},
and prove that $(p_t,\uu_t)$ solve \eqref{eqn:SGsystem2}.

This has been recently done in \cite{ACDF1}, where the following result has been proved (see \cite{ACDF2} for an extension of this result to three dimensions):

\begin{theorem}\label{thm:mainSG}
Let $\bar p : \R^2 \to \R$ be a $\Z^2$-periodic function such that
$\bar p(x)+ |x|^2/2$ is convex, and assume that the measure $
(Id+\nabla \bar p)_\sharp \mathcal L_{\T^2}$ is absolutely continuous with
respect to the Lebesgue measure with density $\bar \rho$, namely
\[
(Id+\nabla\bar  p)_\sharp\mathcal L_{\T^2} =\bar\rho.
\]
Moreover, let us assume \(\lambda \le \bar \rho\le 1/\lambda
\) for some \(\lambda>0\).

Let $\rho_t$ be a solution of \eqref{eqn:dualsystem} starting from $\bar \rho$,
and let $P_t:\R^2\to \R$ be the (unique up
to an additive constant) convex function such that $(\nabla
P_t)_\sharp\mathcal L_{\T^2}=\rho_t$ and $P_t(x)-|x|^2/2$ is $\Z^2$-periodic~\footnote{The existence of such a map is a consequence of  Theorem \ref{thm:McCann}, see \cite[Theorem 2.1]{ACDF1}}. Denote by
$P_t^*:\R^2\to\R$ its convex conjugate.

Then the couple $(p_t,\uu_t)$ defined in \eqref{eqn:velocity} is a distributional solution of
\eqref{eqn:SGsystem2}.
\end{theorem}

Let us point out that \eqref{eq:bounds rhot} is essentially the only regularity property for solutions of \eqref{eqn:dualsystem} that is stable in time. Indeed, even if we start from a smooth initial datum \(\bar \rho\), due to the non-linear coupling between the density \(\varrho_t\) and the velocity field \(\boldsymbol{\mathcal U_t}\), any ``simple''  higher order estimate blows up in finite  time (it is however possible to prove  \emph{short time} existence for smooth solutions, see \cite{LoeSG}). In view of these considerations, it is clear that understanding the  regularity of  solutions of the Monge-Amp\`ere equation  with right hand side just bounded away from zero and infinity it is a key step for the proof of Theorem \ref{thm:mainSG}.

\medskip    \section{The optimal transport problem}
The Monge transportation problem is more than 200 years old
\cite{Mon}, and it has generated in the last years a huge amount
of work.

Originally Monge wanted to move, in the Euclidean space $\R^3$,  a
rubble (\emph{d\'eblais}) to build up a mound or fortification
(\emph{remblais}) minimizing the transportation cost. 
In Monge's original formulation, the cost to transport some mass $m$ from $x$ to $y$
was given by $m|x-y|$. However, it makes sense to consider more general cost functions $c$,
where $c(x,y)$ represents the cost to move a unit mass from $x$ to $y$.

Hence, nowadays, the optimal transport problem is formulated in the
following general form: given two probability measures $\mu$ and
$\nu$ (representing respectively the rubble and the mound) defined on the measurable spaces $X$ and $Y$, find a
measurable map $T:X \rightarrow Y$ with $T_\sharp \mu=\nu$, i.e.,
\begin{equation}\label{pushforward}
\nu(A)=\mu\left(T^{-1}(A)\right) \qquad\forall \,A \subset Y
\mbox{ measurable,}
\end{equation}
in such a way that $T$ minimizes the transportation cost. This
means
\begin{equation}
\label{eq:transport problem}
\int_X c(x,T(x))\,d\mu(x) = \min_{S_\# \mu=\nu} \left\{ \int_X
c(x,S(x))\,d\mu(x) \right\},
\end{equation}
where $c:X \times Y \rightarrow \R$ is some given cost function,
and the minimum is taken over all measurable maps $S:X\to Y$ such
that $S_\# \mu=\nu$. When the transport condition $T_\# \mu=\nu$
is satisfied, we say that $T$ is a \textit{transport map}, and if
$T$ also minimizes the cost we call it an \textit{optimal
transport map}.\\

Even in Euclidean spaces with the cost $c$ equal to the Euclidean
distance or its square, the problem of the existence of an optimal
transport map is far from being trivial. Moreover, it is easy to
build examples where the Monge problem is ill-posed simply because
there is no transport map: this happens for instance when $\mu$ is
a Dirac mass while $\nu$ is not. This means that one needs some
restrictions on the measures $\mu$ and $\nu$.

We notice that,
if \(X,Y\subset \R^n\), $\mu(dx)=f(x)dx$, and $\nu(dy)=g(y)dy$,
if $S:X\to Y$ is a sufficiently smooth transport map one can rewrite the transport condition
$S_\# \mu =\nu$
as a Jacobian equation.
Indeed, if $\chi:\R^n\to \R$ denotes a test function, the condition $S_\# \mu =\nu$
gives
$$
\int_{\R^n}\chi(S(x))f(x)\,dx= \int_{\R^n}\chi(y)g(y)\,dy.
$$
Now, assuming in addition that $S$
is a diffeomorphism, we can set 
 $y=S(x)$ and use the change of variable formula to obtain that the second integral is equal to
$$
\int_{\R^n} \chi(S(x)) g(S(x)) \bigl|\det(\n S(x)) \bigr|\,dx.
$$
By the arbitrariness of $\chi$, this gives the Jacobian equation
$$
f(x)=g(S(x)) \bigl|\det(\n S(x)) \bigr| \qquad \text{a.e.},
$$
as desired.

\medskip    \subsection{The quadratic cost on $\R^n$}
\label{sect:quadratic}

In \cite{Bre1,Bre2}, Brenier considered the case $X=Y=\R^n$ and
$c(x,y)=|x-y|^2/2$, and proved the following theorem
(which was also obtained independently by Cuesta-Albertos and
Matr\'an \cite{CAM} and by Rachev and R\"uschendorf \cite{RR}).
For an idea of the proof, see the sketch of the proof of Theorem \ref{cbrenier}
below, which includes this result as a special case.
\begin{theorem}
\label{thm:Brenier} Let $\mu$ and $\nu$ be two compactly supported
probability measures on $\R^n$. If $\mu$ is absolutely continuous
with respect to the Lebesgue measure, then:
\begin{enumerate}
\item[(i)] There exists a unique solution $T$ to the optimal transport
problem with cost $c(x,y)=|x-y|^2/2$.
\item[(ii)] The optimal map $T$ is characterized by the
structure $T(x)=\nabla u(x)$ for some convex function
$u:\R^n\to \R$.
\end{enumerate}
Furthermore, if $\mu(dx)=f(x)dx$ and $\nu(dy)=g(y)dy$, then $T$ is differentiable $\mu$-a.e. and
\begin{equation}\label{eq:jacbr}
\bigl|\det(\n T(x)) \bigr|=\frac{f(x)}{g(T(x))} \qquad \text{for
$\mu$-a.e. $x \in \R^n$.}
\end{equation}
\end{theorem}

A remark which will be useful later is the following: the cost $|x-y|^2/2$
is equivalent to the cost $-x\cdot y$. Indeed, for any transport map $S$ we have
$$
\int_{\R^n}\frac{|S(x)|^2}2\,d\mu(x)=\int_{\R^n}\frac{|y|^2}2\,d\nu(y)
$$
(this is a direct consequence of the condition $S_\# \mu=\nu$), hence
\begin{align*}
\int_{\R^n}\frac{|x-S(x)|^2}2\,d\mu(x)&=\int_{\R^n}\frac{|x|^2}2\,d\mu(x)
+ \int_{\R^n}\frac{|S(x)|^2}2\,d\mu(x) + \int_{\R^n}\bigl(-x\cdot S(x) \bigr)\,d\mu(x)\\
&=\int_{\R^n}\frac{|x|^2}2\,d\mu(x)
+ \int_{\R^n}\frac{|y|^2}2\,d\nu(y) + \int_{\R^n}\bigl(-x\cdot S(x) \bigr)\,d\mu(x)
\end{align*}
and since the first two integrals in the right hand side are independent of $S$
we see that the two minimization problems
$$
 \min_{S_\# \mu=\nu}  \int_{\R^n}\frac{|x-S(x)|^2}2\,d\mu(x)
\qquad \text{and}\qquad 
 \min_{S_\# \mu=\nu} \int_{\R^n}\bigl(-x\cdot S(x) \bigr)\,d\mu(x)
$$
are equivalent.

\medskip    \subsection{Regularity theory for the quadratic cost: Brenier vs Alexandrov solutions}
\label{sect:Bre Alex}

Before starting the discussion of the regularity of optimal transport maps, let us recall some fact about second order properties of convex functions (see for instance \cite[Theorem 14.25]{Vil}).

\begin{theorem}[Alexandrov]
\label{thm:Alex convex}
Let \(\Omega\) be a convex open set and let \(u:\Omega \to \R\) be a convex function. Then, for a.e. \(x\) in \(\Omega\), $u$ is differentiable at $x$ and there exists symmetric matrix \(D^2 u(x)\) such that
\[
u(y)=u(x)+\nabla u (x)\cdot (y-x)+\frac 1 2 D^2 u(x)(y-x)\cdot (y-x)+o(|y-x|^2).
\]
In addition, at such points $\nabla u$ is differentiable with gradient given by $D^2u$, that is
$$
\nabla u(y)=\nabla u(x)+D^2u(x)\cdot (y-x) +o(|y-x|)\qquad \forall\,y \in {\rm Dom} (\nabla u),
$$
where ${\rm Dom} (\nabla u)$ is the set of differentiability points of $u$.
\end{theorem}

Clearly any convex function admits also a distributional Hessian \(D_\DD^2 u\). Recalling that a positive distribution is a measure, it is simple
to show that \(D_\DD^2u\) is a matrix valued measure \cite[Chapter 6]{EG}. Then one can show that the ``pointwise''  Hessian \(D^2 u\) defined in the Alexandrov theorem is actually the density of the absolutely continuous part of \(D_\DD^2u\) with respect to the Lebesgue measure, i.e.,
$$
D_\DD^2 u=D^2 u\,dx+(D_\DD^2 u)^s.
$$
\smallskip

Let $X$ and $Y$ be two bounded smooth open sets in $\R^n$, and
let $\mu(dx)=f(x)dx$ and $\nu(y)=g(y)dy$ be two probability measures
with $f$ and $g$ such that $f=0$ in $\R^2\setminus X$ and $g=0$ in
$\R^2\setminus Y$. We assume that $f$ and $g$ are
bounded away from zero and infinity on $X$ and $Y$,
respectively. By Brenier's Theorem, when the cost is given by $|x-y|^2/2$ then
the optimal transport map $T$
is the gradient of a convex function $u$. Hence, the Jacobian equation \eqref{eq:jacbr}
combined with the fact that $\nabla u$ is differentiable a.e. (see Theorem \ref{thm:Alex convex})
gives that $u$ solves the Monge-Amp\`ere equation
\begin{equation}
\label{eq:MA}
\det(D^2u(x))=\frac{f(x)}{g(\n u(x))} \qquad f\,dx\text{-a.e.}
\end{equation}
coupled with the ``boundary condition''
\begin{equation}
\label{eq:MA bdry}
\n u(X)=Y
\end{equation}
(which corresponds to the fact that $T$ transports $f(x)dx$ onto
$g(y)dy$).
We will say that the function $u$ is a \emph{Brenier solution} of the Monge-Amp\`ere equation.\\

As observed by Caffarelli \cite{Caf3}, even for smooth densities
one cannot expect any general regularity result for $u$ without
making some geometric assumptions on the support of the target
measure. Indeed, let $n=2$ and suppose that $X=B_1$ is the unit ball centered
at the origin and $Y=\bigl(B_1^+ + e_1\bigr) \cup \bigl(B_1^- -
e_1\bigr)$ is the union of two half-balls (here
$(e_1,e_2)$ denote the canonical basis of $\R^2$), where
$$
B_1^+:=\bigl(B_1 \cap \{x_1
>0\}\bigr),\qquad B_1^-:=\bigl(B_1 \cap \{x_1 <0\}\bigr).
$$
Then, if $f=\frac{1}{|X|}\mathbf 1_X$ and $g=\frac{1}{|Y|}\mathbf 1_Y$, it
is easily seen that the optimal map $T$ is given by
$$
T(x):=\left\{ \begin{array}{ll} x+e_1& \text{if } x_1 >0,\\
x-e_1 &\text{if }x_1<0,\end{array}\right.
$$
which corresponds to the gradient of the convex function
$u(x)= |x|^2/2+|x_1|$.

Thus, as one could also show by an easy topological argument, in
order to hope for a regularity result for $u$ we need at least
to assume the connectedness of $Y$. But, starting from the above
construction and considering a sequence of domains $X_\e'$ where
one adds a small strip of width $\e>0$ to glue together
$\bigl(B_1^+ + e_2\bigr) \cup \bigl(B_1^- - e_2\bigr)$, one can
also show that for $\e>0$ small enough the optimal map will still
be discontinuous (see \cite{Caf3}).\\

The reason for this lack of regularity is the following: 
For Alexandrov solutions the multivalued map $x \mapsto \partial u(x)$
preserves the Lebesgue measure up to multiplicative constants, that is, the volumes
of $E$ and $\partial u(E)$ are comparable for any Borel set $E\subset X$.

On the other hand, for Brenier solutions, \eqref{eq:MA} gives the same kind of information but only at points where $u$
is twice differentiable, so \eqref{eq:MA} may miss some singular part in the Monge-Amp\`ere measure.
This comes from the fact that the optimal map can only see the regions where $f$ and $g$ live: roughly speaking, for Brenier solutions we only have
$$
|E|\simeq |\partial u(E)\cap Y| \qquad \forall\, E \subset X
$$
(and not $|E|\simeq |\partial u(E)|$ as in the Alexandrov case),
which means that we do not  have a full control on the Monge-Amp\`ere measure of \(u\).
\smallskip

In the counterexample above what happens is the following:
when $x=(0,x_2)$ then $\partial u(x)=[-1,1]\times \{x_2\}$, hence $\partial u$ maps the segment
$\{0\}\times [-1,1]$
onto the square $[-1,1]^2$, but in the latter square $g$ has no mass.
Notice that in this case the determinant of $D^2u$ is equal to $1$ a.e. inside $B_1$ (so $u$ is a Brenier
solution with right hand side $1$) but the Monge-Amp\`ere measure of $u$ is equal to
$dx\res B_1+\mathcal H^{1}\res (\{x_{1}=0\}\cap B_1)$
 (compare with Example \ref{ex:disc}).
\smallskip

Hence, in order to avoid this kind of counterexample one should make sure that the target mass always 
cover the image of $\partial u(X)$, and a way to ensure this is that $Y$ is convex.
Indeed, as shown by Caffarelli \cite{Caf3}, if $Y$ is convex then $\partial u(X)\subset \overline Y$
(see the proof of Theorem \ref{thm:regularity transport} below) and any Brenier solution is an Alexandrov solution.
In particular the regularity results from Section \ref{sect:Caff Alex}
apply whenever $f$ and $g$ are strictly positive on their respective support
(since, in that case, the right hand side in \eqref{eq:MA} is bounded away from zero and infinity)
and we have the following (see also \cite{Caf3,Caf4}):
\begin{theorem}
\label{thm:regularity transport}
Let $X,Y\subset \R^n$ be two bounded open sets, let
$f:X\to \R^+$ and $g:Y\to \R^+$ be two
probability densities, respectively 
bounded away from zero and infinity on $X$ and $Y$,
and denote by $T=\nabla u:X\to Y$ the unique optimal transport map sending $f$ onto $g$
for the cost $|x-y|^2/2$. Assume that $Y$ is convex.
Then:
\begin{enumerate}
\item[-]
$T \in C^{0,\alpha}_{\rm loc}(X)\cap W^{1,1+\e}_{\rm loc}(X)$.
\item[-] If in addition $f \in C^{k,\beta}_{\rm loc}(X)$ and $g\in C^{k,\beta}_{\rm loc}(Y)$ for some $\beta \in (0,1)$, then
$T \in C^{k+1,\beta}_{\rm loc}(X)$.
\item[-] Furthermore, if $f \in C^{k,\beta}(\overline X)$, $g\in C^{k,\beta}(\overline Y)$, and both $X$ and $Y$ are smooth and uniformly convex, then $T:\overline X\to\overline Y$ is a global diffeomorphism of class $C^{k+1,\beta}$.
\end{enumerate}
\end{theorem}

\begin{proof}[Sketch of the proof] As we explained, at least for the interior regularity, the key step in the proof is to show that when \(Y\) is convex Brenier solution are Alexandrov solutions. Let us briefly sketch the proof of this fact.\\

\noindent
\emph{Step 1: For any set \(A\subset X\) it holds
\[
\mu_u(A)\ge \int_A \frac {f(x)}{g(\nabla u(x)} \,dx.
\]}
This is a general fact that does not need the convexity of \(Y\). Indeed, for any set \(A\subset X\),
\[
\partial u(A)\supset \nabla u \big(A\cap {\rm Dom} (\nabla u)\big).
\]
Then, since by Alexandrov Theorem \(\nabla u\) is differentiable almost everywhere, by the Area Formula  \cite[Corollary 3.2.20]{Federer} and \eqref{eq:jacbr} we get
\[
\mu_u (A)=|\partial u (A)|\ge \big|\nabla u \big(A\cap {\rm Dom} (\nabla u)\big)\big|=\int_A \det D^2 u\,dx=\int_A \frac {f(x)}{g(\nabla u(x)} \,dx.
\]
\\

\noindent
\emph{Step 2:  If \(\partial u(A) \subset Y \) up to a set of measure zero,  then 
\[
\mu_u (A)=\int_A \frac {f(x)}{g(\nabla u(x)} \,dx.
\]}
To see this notice that, for all \(A\subset X\),
\[
A\cap {\rm Dom} (\nabla u) \subset (\nabla u)^{-1}(\p u(A)) 
\]
and 
\[
\begin{split}
(\nabla u)^{-1}&\big(\p u (A)\cap Y\big)\setminus A \\
 &\subset(\nabla u)^{-1}\Big(  \big\{y \in Y\,:\, \text{there exist \(x_1,x_2\), \(x_1\ne x_2\) such that \(y\in \p u(x_1)\cap \p u(x_2)\)}\big\}\Big)\\
&\subset (\nabla u)^{-1} \Big(\big\{\textrm{Points of non-differentiability of \(u^*\)}\big\}\cap Y\Big)
\end{split}
\]
where \(u^*\) is the convex conjugate of \(u\), see for instance \cite[Lemma 1.1.12]{G}.  Since any convex function is differentiable almost everywhere, by our assumptions on the densities and the transport relation \((\nabla u)_{\sharp}(f dx)=g dy\) we infer that
\[
\big| (\nabla u)^{-1}\big(\p u (A)\cap Y\big)\cap (X \setminus A)\big|=0.
  \] 
Since \(\partial u(A) \subset Y\) a.e.  and \(f\) vanishes outside \(X\), using again that \((\nabla u)_{\sharp}(f dx)=g dy\) we get
\[
\begin{split}
|\p u (A)|&= \int_{\p u(A) \cap Y}\frac{g(y)}{g(y)}\,dy=\int_{(\nabla u)^{-1}(\p u (A)) }\frac{f(x)}{g(\nabla  u (x))}\,dx\\
&= \int_{A }\frac{f(x)}{g(\nabla  u (x))}\,dx+\int_{(\nabla u)^{-1}(\p u (A))\setminus A  }\frac{f(x)}{g(\nabla  u (x))}\,dx=  \int_{A }\frac{f(x)}{g(\nabla  u (x))}\,dx.\\
\end{split}
\]
\\

\noindent
{\em Step 3: $\partial u(X)\subset Y$.} Recall that, as a general fact for convex functions,
  \begin{equation}\label{singapore}
  \partial  u (x)={\rm Conv}\Big(\big\{\textrm{\(p\) \,:\, there exist \(x_k\in {\rm Dom}(\nabla u)\) with \(x_k \to x\) and \(\nabla u(x_k)\to p\)}\big\}\Big), 
\end{equation}
see \cite[Theorem 3.3.6]{CanSin}. Since \((\nabla u)_\sharp(f\,dx)=g\,dy\), the set of \(x\in X\cap {\rm Dom} (\nabla u )\) such that \(\nabla u(x)\in Y\) is of full measure in \(X\), in particular it is dense. From this,  \eqref{singapore}, and the convexity of $Y$, one immediately deduces that
\[
\partial u (X)\subset \overline{ {\rm Conv}(Y)}=\overline Y.
\]
Since the measure of \(\partial Y\) is zero, using  the previous step  we conclude that \(u\) is an Alexandrov solution of
\[
\det D^2 u = \frac {f}{g\circ\nabla u}.
\]
Finally, to apply the results from Section \ref{sect:Caff Alex} one has then to show that \(u\) is strictly convex inside \(X\), see \cite{Caf3} for more details.
\end{proof}

A natural question is what happens when one removes the convexity assumption on the target.
As shown in \cite{FK-partial} (see also \cite{F-partial} for a more precise description of the singular set in two dimensions),
in this case one can prove that the optimal trasport map is smooth outside
a closed set of measure zero. More precisely, the following holds:

\begin{theorem}
\label{thm:partialMA}
Let $X,Y\subset \R^n$ be two bounded open sets, let
$f:X\to \R^+$ and $g:Y\to \R^+$ be two
probability densities, respectively 
bounded away from zero and infinity on $X$ and $Y$,
and denote by $T=\nabla u:X\to Y$ the unique optimal transport map sending $f$ onto $g$
for the cost $|x-y|^2/2$.
Then there exist two relatively closed sets $\Sigma_X \subset X,\Sigma_Y\subset Y$ of measure zero such that
$T:X\setminus \Sigma_X \to Y\setminus \Sigma_Y$ is a homeomorphism of class $C_{\rm loc}^{0,\alpha}$ for some $\alpha>0$.
In addition, if $c\in C^{k+2,\alpha}_{\rm loc}(X\times Y)$, $f\in C^{k,\alpha}_{\rm loc}(X)$, and $g\in C^{k,\alpha}_{\rm loc}(Y)$
for some $k \geq 0$ and $\alpha \in (0,1)$, then $T:X\setminus \Sigma_X\to Y\setminus \Sigma_Y$ is a
diffeomorphism of class $C_{\rm loc}^{k+1,\alpha}$.
\end{theorem}

\begin{proof}[Sketch of the proof] As explained above, when $Y$ is not convex there could points
$x \in X$ such that \(\partial u(x)\nsubseteq Y\) and on these points we have no control on the Monge-Amp\`ere measure of \(u\). Let us  define \footnote{Actually, in \cite{F-partial, FK-partial} the regular set is defined in a slightly different way and it is in principle smaller. However,
the advantage of that definition is that it allows for a more refined analysis
of the singular set (see \cite{F-partial}).}
\[
{\rm Reg}_{X}:=\{x\in X:\ \partial u (x)\subset Y\}\qquad \Sigma_X:=X\setminus {\rm Reg}_{X}.
\]
By the continuity property of the subdifferential  \cite[Proposition 3.3.4]{CanSin}, it is immediate to see that \({\rm Reg}_{X}\) is open. Moreover it follows from the condition \((\nabla u)_\sharp (f\,dx)=g\, dy\) that $\nabla u(x) \in Y$ for a.e. $x \in X$, hence \(|\Sigma_X|=0\). Thus,
arguing as in Step 2 in the proof of Theorem \ref{thm:regularity transport} one can show  that \(u\) is a strictly convex Alexandrov solution on \({\rm Reg}_{X}\), so one can apply the regularity results of Section \ref{sect:classical MA}.
\end{proof}

\medskip
\subsection{The partial transport problem}

Let us remark that, in addition to the situation described in Section \ref{semigeo},
the regularity theory for optimal transport maps with right hand side only bounded away from zero and infinity appears in other situations.
An important example is provided
by the ``partial transport problem'' described here.\\

Given two densities $f$ and $g$, one
wants to transport a fraction $m \in
[0,\min\{\|f\|_{L^1},\|g\|_{L^1}\}]$ of the mass of $f$ onto $g$
minimizing the transportation cost $c(x,y)=|x-y|^2/2$. 
The existence and uniqueness of minimizers it is a very nontrivial problem.

In particular, one can notice that in
general uniqueness of minimizers does not holds: if $m \leq
\int_{\mathbb{R}^n} \min\{f, g\}$, the family of minimizers is simply given by choosing
any fraction of mass of  $\min\{f, g\}$ with total volume $m$
and just send this mass to itself with zero cost.\\

To ensure uniqueness, in \cite{cafmcc}
the authors assumed $f$ and $g$ to have disjoint supports. Under
this assumption they were able to prove (as in the classical optimal transport problem) that there exists a convex
function $u$ such that the unique minimizer is given by
$\nabla u$. This function $u$ is also shown to solve in a
weak sense a ``Monge-Amp\`ere double obstacle problem''. Then,
strengthening the disjointness assumption into the hypothesis on
the existence of a hyperplane separating the supports of the two
measures, the authors prove a semiconvexity result on the free
boundaries. Furthermore, under some classical regularity
assumptions on the measures and on their supports, local
$C^{1,\alpha}$ regularity of $u$ and on the free boundaries of
the active regions is shown.\\

In \cite{F-partial transport} the second author studied what happens if one removes the
disjointness assumption. Although minimizers are non-unique for
$m<\int_{\mathbb{R}^n} \min\{f,g\}$ (but in this case the set of
minimizers can be trivially described, as mentioned above), uniqueness holds for any
$m\geq \int_{\mathbb{R}^n} \min\{f,g\}$. Moreover, exactly as in
\cite{cafmcc}, the unique minimizer is given gradient of a convex function.

An important fact proved in \cite{F-partial transport} is that both the source and the target mass
dominate the common mass $\min\{f,g\}$ (that is, the common mass
is both source and target). This property, which has an interest
on its own, implies that the free boundaries never enter inside the intersection of the supports
of $f$ and $g$, a fact which plays a crucial role in the regularity of the
free boundaries. Indeed, it is shown that the free boundaries have
zero Lebesgue measure (under some mild assumptions on the supports
of the two densities). As a consequence of this fact, whenever the support of $g$ is assumed to be convex and
$f$ and $g$ are bounded away from zero and infinity on their
respective support,  Theorem \ref{thm:regularity transport} allows one to deduce local $C^{0,\alpha}$
regularity of the transport map, and to show that it extends to an
homeomorphism up to the boundary if both supports are assumed to
be strictly convex. Further regularity results on the free boundaries were later obtained in \cite{F-note partial transport} and \cite{Indrei}.

On the other hand, in this situation where the supports of $f$ and
$g$ can intersect, something new happens: in the classical optimal transport problem, by assuming
$C^\infty$ regularity on the density of $f$ and $g$ (together with
some convexity assumption on their supports) Theorem \ref{thm:regularity transport} ensures that
the transport map is $C^\infty$ too. In contrast with this, in \cite{F-partial transport}
the author showed that
$C^{0,\alpha}_{\rm loc}$ regularity is optimal: one
can find two $C^\infty$ densities on $\mathbb{R}$, supported on
two bounded intervals and bounded away from zero on their
supports, such that the transport map is not $C^1$.

\medskip    \subsection{The case of a general cost}\label{gcsect}

After Theorem \ref{thm:Brenier}, many researchers started to work on the
problem of showing existence of optimal maps in the case of more general
costs, both in an Euclidean setting and in the case of
Riemannian manifolds. Since at least locally any Riemannian manifold looks like \(\R^n\),
here we shall focus on the case of $\R^n$ and then in Section \ref{sect:Riem}
we will discuss some of the nontrivial issues that may arise on manifolds (for instance, due to 
the presence of the cut-locus).\\


We introduce first some conditions on the cost.
Here and in the sequel, $X$ and $Y$ denote two open subsets of $\R^n$.
\begin{enumerate}
\item[{\bf (C0)}] The cost function $c:X\times Y\to \R$ is of class $C^{4}$
with $\|c\|_{C^{4}(X\times Y)}<\infty$.
\item[{\bf (C1)}] For any $x \in X$, the map $Y \ni y \mapsto -D_x c(x, y) \in \R^n$ is injective.
\item[{\bf (C2)}] For any $y \in Y$, the map $X \ni x \mapsto -D_y c(x, y) \in \R^n$ is injective.
\item[{\bf (C3)}] $\det(D_{xy} c)(x,y) \neq 0$ for all $(x,y) \in X \times Y$.
\end{enumerate}

We now introduce the notion of \emph{$c$-convexity}, which will play an important role in all
our discussion.

A function $\psi:X \to \R\cup\{+\i\}$ is
\emph{$c$-convex} if \footnote{
The terminology comes from the fact that, in the case $c(x,y)=-x\cdot y$
(which is equivalent to $|x-y|^2/2$, see Section \ref{sect:quadratic}), $c$-convexity corresponds to the classical notion of convexity (since in that case $\psi$ is a supremum of linear functions). Analogously, the $c$-subdifferential defined in \eqref{eq:csubdiff} corresponds to the classical subdifferential for convex functions.}
\begin{equation}
\label{eq:cconvex} 
\psi(x)=\sup_{y\in Y} \bigl[\psi^c(y)-c(x,y)\bigr]\qquad
\forall\,x \in X,
\end{equation}
where $\psi^c:Y\to \R\cup \{-\infty\}$ is given by
$$
\psi^c(y):=\inf_{x\in X} \bigl[\psi(x)+c(x,y)\bigr]\qquad \forall
\,y \in Y.
$$
For a $c$-convex function $\psi$, we define its
\emph{$c$-subdifferential} at $x$ as
$$
\p_c\psi(x):=\bigl\{ y\in Y \,:\,\psi(x)=\psi^c(y)-c(x,y) \bigr\},
$$
or, equivalently,
\begin{equation}
\label{eq:csubdiff}
\p_c \psi(x) := \{y \in \overline Y\, : \ \psi(z)\geq -c(z,y)+c(x,y)+\psi(x) \quad \forall\,z \in X\}.
\end{equation}
 We also define the \textit{Fr\'echet subdifferential }of $\psi$ at $x$ as
\begin{equation}
\label{Frechet subdiff}
\p^- \psi(x) := \bigl\{p \in \R^n \,: \ \psi(z) \geq u(x)+p \cdot (z-x)+o(|z-x|)\bigr\}.
\end{equation}
It is easy to check that, if $c$ is of class $C^1$, then the following inclusion holds:
\begin{equation}
\label{eq:rel subdiff}
y \in \p_c\psi(x) \quad \Longrightarrow\quad -D_xc(x,y) \in \p^- \psi(x).
\end{equation}
 In addition, if $c$ satisfies {\bf (C0)}-{\bf (C2)}, then we can define the \textit{$c$-exponential map}:\footnote{
 The name ``$c$-exponential'' is motivated by the fact that, when $c$ is given by the squared Riemannian distance on a manifold, the $c$-exponential map coincides with the classical exponential map in Riemannian geometry.
 }
\begin{equation}
\label{eq:cexp}
\text{for any $x\in X$, $y \in Y$, $p \in \R^n$},
\qquad
\cexp_x(p)=y \quad \Leftrightarrow \quad p=-D_xc(x,y).
\end{equation}
Using \eqref{eq:cexp}, we can rewrite \eqref{eq:rel subdiff} as
\begin{equation}
\label{eq:rel subdiff2}
\p_c\psi(x) \subset \cexp_x\left(\p^- \psi(x)\right).
\end{equation}
Let us point out that, when $c(x,y)=-x\cdot y$, $\psi$ is convex and $\p_c\psi=\p \psi$, hence the above inclusion is an equality (since any
local supporting hyperplane is a global support).
However, 
for a general cost function $c$,
the above inclusion may be strict
when $\psi$ is non-differentiable.
As we will discuss in Section \ref{MTWsec}, equality in \eqref{eq:rel subdiff2} for every \(c\)-convex function  is necessary for regularity of optimal maps between smooth densities. Notice that, if $c \in C^1$ and $Y$ is bounded,
it follows immediately from \eqref{eq:cconvex} that $c$-convex functions are Lipschitz, so
in particular they are differentiable a.e.

\begin{remark}\label{rmk:semiconv}
If $c$ satisfies {\bf (C0)} and $Y$ is bounded, then 
it follows from \eqref{eq:cconvex}
that $\psi$ is semiconvex (i.e., there exists a constant $C>0$ such that $\psi+C|x|^2/2$ is convex,
see for instance \cite{FF}).
In particular, by Theorem \ref{thm:Alex convex},
$c$-convex functions are twice differentiable a.e.
\end{remark}
The following is a basic result in optimal transport theory.

\begin{theorem}\label{cbrenier}
Let $c:X\times Y \to \R$ satisfy {\bf (C0)}-{\bf (C1)}.
Given two probability densities $f$ and $g$ supported on $X$ and $Y$ respectively,
there exists a $c$-convex function $u:X\to \R$ such that the map $T_u:X\to Y$ defined by $T_u(x):= \cexp_x(\n u(x))$ is the unique optimal transport map sending $f$ onto $g$.
If in addition {\bf (C2)} holds, then $T_u$ is injective $f\,dx$-a.e.,
$$
\bigl|\det(\n T_u(x)) \bigr|=\frac{f(x)}{g(T_u(x))} \qquad \text{$f\,dx$-a.e.},
$$
and its inverse is given by the optimal
transport map sending $g$ onto $f$.
\end{theorem}
Notice that, in the particular case $c(x,y)=-x\cdot y$ (which is equivalent to the quadratic cost $|x-y|^2/2$, see Section \ref{sect:quadratic}), $c$-convex functions
are convex and the above result corresponds to Theorem \ref{thm:Brenier}.
We give here a sketch of the proof of Theorem \ref{cbrenier}, referring to \cite[Chapter 10]{Vil}
for more details.

\begin{proof}[Sketch of the proof]
There are several ways to establish this result. One possibility is to go through the
following steps:\\

\noindent
\textit{Step 1: Solve the Kantorovich problem.} Following Kantorovich \cite{kant1,kant2}, 
we consider the following relaxed problem:
instead of minimizing the transportation cost among all transport maps (see \eqref{eq:transport problem}), we consider instead the problem
\begin{equation}
\label{kantprob}
\inf_{\pi \in \Pi(f,g)} \left\{ \int_{X \times Y} c(x,y)\,d\pi(x,y) \right\},
\end{equation}
where $\Pi(f,g)$ denotes the set of all probability measures $\pi$ on $X\times Y$
whose marginals are $f\,dx$ and $g\,dy$, i.e.,
$$
\int_{X \times Y} h(x) \,d\pi(x,y)=\int_X h(x)f(x)\,dx, \qquad \int_{X \times Y}
h(y) \,d\pi(x,y)=\int_Y h(y)g(y)\,dy,
$$
for all $h:M \to \R$ bounded continuous.

The connection between the formulation of Kantorovich and that of Monge is the following: any transport map
$T$ induces the plan defined by $(\operatorname{Id} \times T)_\# \mu$ which is concentrated on
the graph of $T$.
Conversely, if a transport plan is concentrated on the graph of a measurable function $T$,
then it is induced by this map.

By weak compactness of the set $\Pi(f,g)$ and continuity of the function
$\pi \mapsto \int c(x,y)\,d\pi$, it is simple to show
the existence of a minimizer $\bar\pi$ for \eqref{kantprob};
so to prove the existence of a solution to the Monge problem it suffices to show that $\bar\pi$ is
concentrated on the graph of a measurable map $T$, i.e.,
$$
y=T(x) \qquad \text{for $\bar\pi$-almost every $(x,y)$}.
$$
Once this fact is proved, the uniqueness of optimal maps will follow from the observation that,
if $T_1$ and $T_2$ are optimal, then $\pi_1:=(\operatorname{Id} \times T_1)_\# \mu$
and $\pi_2:=(\operatorname{Id} \times T_2)_\# \mu$ are both optimal plans, so
by linearity $\bar \pi=\frac{1}{2}(\pi_1+\pi_2)$ is optimal.
If it is concentrated on a graph, this implies that $T_1$ and $T_2$ must coincide $f\,dx$-a.e.,
 proving the desired uniqueness.\\

\noindent
\textit{Step 2: The support of $\bar\pi$ is $c$-cyclically monotone.}
A set $S\subset X \times Y$ is called \textit{$c$-cyclically monotone} if, for all $N \in \N$,
for all $\{(x_i,y_i)\}_{0\leq i \leq N} \subset S$, one has
$$
\sum_{i=0}^N c(x_i,y_i) \leq \sum_{i=0}^N c(x_i,y_{i+1}),
$$
where by convention $y_{N+1}=y_0$.
The above definition heuristically means that, sending the point $x_i$ to the point $y_i$ for $i=0,\ldots,N$
is globally less expensive than sending the point $x_{i}$ to the point $y_{i+1}$.
It is therefore intuitive that, since $\bar\pi$ is optimal, its support is $c$-cyclically monotone
(see \cite{gangbomccann} or \cite[Chapter 5]{Vil} for a proof).\\

\noindent
\textit{Step 3: Any $c$-cyclically monotone set is contained in the $c$-subdifferential of a $c$-convex function.}
A proof of this fact (which is due to Rockafellar for $c(x,y)=-x\cdot y$, and R\"uschendorf for the
general case) consists in constructing explicitly a $c$-convex function which does the work:
given $S$ $c$-cyclically monotone, we define
\begin{multline*}
u(x):=\sup_{N \in \N}\sup_{\{(x_i,y_i)\}_{1\leq i \leq N} \subset S}
\Bigl\{\bigl[c(x_0,y_0)-c(x_1,y_0)\bigr]+\bigl[c(x_1,y_1)-c(x_2,y_1)\bigr]\\
+\ldots+\bigl[c(x_N,y_N)-c(x,y_N)\bigr] \Bigr\},
\end{multline*}
where $(x_0,y_0)$ is arbitrarily chosen in $S$.
It can be easily checked that, with this definition, $u$ is $c$-convex
and $S\subset  \cup_{x \in X}\bigl(\{x\}\times \p_cu(x)\bigr)$
(see for instance \cite[Proof of Theorem 5.10]{Vil}).\\

\noindent
\textit{Step 4: $\bar \pi$ is concentrated on a graph.}
Applying Steps 2 and 3, we know that the support of $\bar\pi$ is contained in the $c$-subdifferential of a $c$-convex function $ u$.
Moreover, being $u$ the supremum of the family of uniformly Lipschitz functions $c(\cdot,y)+\lambda_y$ (see \eqref{eq:cconvex}), it is Lipschitz, hence differentiable a.e.
Being the first marginal of $\bar\pi$ absolutely continuous with respect to the Lebesgue measure,
we deduce that, for $\bar \pi$-almost every $(x,y)$, $ u$ is differentiable at $x$.

Now, let us fix a point $(\bar x,\bar y) \in \supp(\pi)$ such that $ u$ is differentiable at $\bar x$. To prove that $\bar\pi$
is concentrated on a graph, it suffices to prove that $\bar y$ is uniquely determined as a function of $\bar x$.
To this aim, we observe that, since the support of $\bar \pi$ is contained in the $c$-subdifferential of $u$, we have
$\bar y \in \p_c u(\bar x)$, and this implies that the function $x \mapsto  u(x)+c(x,\bar y)$
attains a minimum at $\bar x$. Hence, being $ u$ is differentiable at $\bar x$ we obtain
$$
\n  u(\bar x)+\n_xc(\bar x,\bar y)=0.
$$
Recalling \eqref{eq:cexp} this implies that
$$
\bar y=\cexp_{\bar x}(\n u(\bar x))=:T_u(\bar x).
$$
This proves that $\bar \pi$ is concentrated on the graph of the map $T_u$, as desired.\\

This concludes the first part of the theorem. The invertibility of $T_u$ under ${\bf (C2)}$
follows by a simple argument based on the uniqueness of the minimizer of \eqref{kantprob},
see \cite[Remark 6.2.11]{AGS} for more details.
\end{proof}

In the case when $X=Y=(M,g)$ is a compact Riemannian manifold one has to face the additional difficulty that the cost function $c=d^2/2$ is not smooth everywhere.
However, by exploring the superdifferentiability properties of the squared Riemannian distance function
(i.e., the fact that, for any $y \in M$, the Fr\'echet 
subdifferential of $-d^2(\cdot,y)$ is nonempty at every point) McCann was able to extend the theorem
above the the case of compact Riemannian manifolds \cite{McC}
(see also \cite{FF} for a more general result):

\begin{theorem} \label{thm:McCann}
Let $M$ be a smooth compact Riemannian manifold, and consider the cost $c=d^2/2$,
$d$ being the Riemannian distance.
Given two probability densities $f$ and $g$ supported on $M$,
there exists a $c$-convex function $u:M\to \R\cup\{+\infty\}$ such that
 $T_u(x)=\exp_x(\nabla u(x))$  is the unique optimal transport map sending $f$ onto $g$.
\end{theorem}

\medskip    \section{A class of Monge-Amp\`ere type equations}\label{MTWsec}
As shown in Theorem \ref{cbrenier}, whenever the cost function satisfies  {\bf (C0)}-{\bf (C1)}
then the unique optimal transport map sending $f$ onto $g$ is given by 
$T_u(x)=\cexp_x(\n u(x))$. Furthermore, if $c$ satisfies {\bf (C2)}, then
\begin{equation}
\label{eq:detT}
|\det(\n T_u(x))|=\frac{f(x)}{g(T_u(x))}\qquad \text{a.e.}
\end{equation}
Now, since $\{\cexp_x(\nabla u(x))\}= \partial_cu(x)$
at every differentiability point (see \eqref{eq:rel subdiff2}), it follows that
$z \mapsto  u(z)+c\bigl(z,\cexp_x(\nabla u(x))\bigr)$ attains 
a minimum at $z=x$. Hence, whenever $u$ is twice differentiable at $x$
(that is at almost   every point, see Remark \ref{rmk:semiconv}) we get
\begin{equation}
\label{eq:pos def}
D^2u(x)+D_{xx}c\bigl(x,\cexp_x(\nabla u(x))\bigr) \geq 0.
\end{equation}
Hence, rewriting the relation $T_u(x)=\cexp_x(\n u(x))$ as
$$
-D_xc(x,T_u(x))=\nabla u(x)
$$
and differentiating the above equation with respect to $x$,
 using \eqref{eq:detT} and \eqref{eq:pos def} we obtain 
\begin{equation}
\label{eq:MAcost}
\det\Bigl(D^2u(x)+D_{xx}c\bigl(x,\cexp_x(\nabla u(x))\bigr) \Bigr)=\left|\det\left(D_{xy}c\bigl(x,\cexp_x(\nabla u(x))\bigr) \right) \right| \frac{f(x)}{g(\cexp_x(\nabla u(x)))}
\end{equation}
at every point $x$ where $u$ it is twice differentiable.

Hence $u$ solves a Monge-Amp\`ere type equation of the form
\eqref{eq:MAgeneral} with
$$
\mathcal A(x,\nabla u(x)):=-D_{xx}
c\bigl(x,\cexp_x\bigl(\n u(x)\bigr)\bigr).
$$
Notice that, if {\bf (C3)} holds and $f$ and $g$ are bounded away from zero and infinity on their respective supports, then the right hand side in \eqref{eq:MAcost} is bounded away from zero
and infinity on the support of $f$.
As we will see below, some structural conditions on $\mathcal A$
(or equivalently on the cost $c$) must be imposed to obtain
regularity results.\\

The breakthrough in the study of regularity of optimal transport maps came with the paper of Ma,
Trudinger, and Wang \cite{MTW} (whose roots lie in an earlier work
of Wang on the reflector antenna problem \cite{Wan}), where the
authors found a mysterious fourth-order condition on the cost
functions, which turned out to be sufficient to prove the
regularity of $u$. The idea was to differentiate twice equation
\eqref{eq:MAcost} in order to get a linear PDE for the second
derivatives of $u$, and then to try to show an a priori
estimate on the second derivatives of $u$, compare with Theorem \ref{po1}. In this computationone obtains at a certain moment with a term which needs to have a
sign in order to conclude the desired a priori estimate. This term
is what is now called the ``Ma--Trudinger--Wang tensor'' (in short
$\MTW$ tensor):
\begin{equation}
\label{eq:MTW tensor}
\begin{split}
\Smtw_{(x,y)}(\xi,\eta)&:=
D^2_{p_\eta p_\eta}\mathcal A(x,p) [\xi,\xi]\\
&=
\sum_{i,j,k,l,p,q,r,s}\left(c_{ij,p}c^{p,q}c_{q,rs}
-c_{ij,rs} \right)c^{r,k}c^{s,l}\xi^i\xi^j \eta^k\eta^l, \qquad\xi ,\eta \in
\R^n.
\end{split}
\end{equation}
In the above formula the cost function is evaluated at $(x,y)=(x,\cexp_x(p))$,
and we used the notation $c_{j}=\frac{\p c}{\p x^j}$,
$c_{jk}=\frac{\p^2 c}{\p x^j \p x^k}$, $c_{i,j}=\frac{\p^2 c}{\p
x^i \p y^j}$, $c^{i,j}=(c_{i,j})^{-1}$.  The condition
to impose on $\Smtw_{(x,y)}(\xi,\eta)$ is
$$
\Smtw_{(x,y)}(\xi,\eta) \geq 0 \qquad \text{whenever
}\xi\perp \eta
$$
(this is called the $\MTW$ condition).
Actually, it is convenient to introduce a more general definition:
\begin{definition}
Given $K \geq 0,$ we say that $c$ satisfies the $\MTW(K)$
condition if, for all $(x,y)\in (X\times Y)$ and for all $\xi ,\eta\in \R^n$,
$$
\Smtw_{(x,y)}(\xi,\eta)\geq K|\xi|^2 |\eta|^2 \qquad
\text{whenever }\xi\perp \eta.
$$
\end{definition}
Under this hypothesis, and
a geometric condition on the supports of the measures (which is
the analogous of the convexity assumption of Caffarelli), Ma,
Trudinger, and Wang could prove the following result \cite{MTW,TW1,TW2} (see also \cite{Tr}):

\begin{theorem}
\label{thm:mtw} Let $c:X\times Y \to \R$ satisfy {\bf (C0)}-{\bf (C3)}. 
Assume that $\MTW(0)$ holds, that $f$ and $g$ are
smooth and bounded away from zero and infinity on their respective
supports $X$ and $Y$.
Also, suppose that:
\begin{enumerate}
\item[(a)] $X$ and $Y$ are smooth; \item[(b)]
$D_xc(x,Y)$ is uniformly convex for all $x
\in X$; \item[(c)] $D_yc(X,y)$ is uniformly
convex for all $y \in Y$.
\end{enumerate}
Then $u \in C^\infty(\overline X)$ and $T:\overline X \to
\overline{Y}$ is a smooth diffeomorphism, where $T(x)=\cexp_x(\n u(x))$.
\end{theorem}

\begin{proof}[Sketch of the proof]
For simplicity here we treat the simpler case when $\MTW(K)$ holds for some $K>0$.
This stronger $\MTW$
condition is actually the one originally used in \cite{MTW,TW2}.
The general case $K=0$ is treated in \cite{TW1}, where the
authors relax the stronger assumption by applying a 
barrier argument.

As we pointed out in Remark \ref{rmk:degen} for the classical Monge-Amp\`ere equation  the key point to obtain existence of smooth solutions  was to show an a
priori estimate on second derivatives of solutions. The same is true in the case of \eqref{eq:MAcost}. Indeed, once we know an a priori \(C^2\) bound  on smooth solutions of \eqref{eq:MAcost},
this equation becomes uniformly elliptic (in the sense that the linearized operator is uniformly elliptic, see Remark \ref{rmk:degen}) and one can run a parallel strategy to the one described in Section \ref{sect:continuity} to get the existence of smooth solutions (see \cite{TW1} for more details).
%

Let us start from a smooth (say $C^4$) solution of
\eqref{eq:MAcost}, coupled with the boundary condition $T(X)=Y$,
where $T(x)=\cexp_x(\n u(x))$. Our goal here is to show an interior universal bound for the second
derivatives of $u$.

We observe that, since $T(x)=\cexp_x(\n u(x))$, we have
$$
|\n u(x)|= |\nabla_xc(x,T(x))| \leq C,
$$
hence $u$ is globally Lipschitz, with a uniform Lipschitz
bound. Define
$$
w_{ij}:=D_{x^ix^j} u+ D_{x^ix^j}c\bigl(x,\cexp_x(\n
u(x))\bigr).
$$
(Recall that, by the $c$-convexity of $u$ and \eqref{eq:MAcost}, $(w_{ij})$ is
positive definite.)
Then \eqref{eq:MAcost} can be written
as
\begin{equation}
\label{eq:jacob w} \det(w_{ij})=h(x,\n u(x)),
\end{equation}
or equivalently
$$
\log \bigl( \det(w_{ij})\bigr)=\var,
$$
with $\var(x):=\log\bigl(h(x,\n u(x))\bigr)$. By differentiating
the above equation and using the convention of summation over
repeated indices, we get
$$
w^{ij}w_{ij,k}=\var_k,
$$
$$
w^{ij}w_{ij,kk}=\var_{kk}+ w^{is}w^{jt}w_{ij,k}w_{st,k} \geq
\var_{kk},
$$
where $(w^{ij})$ denotes the inverse of $(w_{ij})$ and we use the
notation $\psi_{k}=\frac{\p}{\p x^k} \psi$, $w_{ij,k}=\frac{\p}{\p
x^k}w_{ij}$, $T_{s,k}=\frac{\p}{\p x^k} T_s$, and so on. Then the
above equations become
\begin{equation}
\label{eq:MTW1} w^{ij}\bigl[u_{ijk} + c_{ijk} + c_{ij,s}
T_{s,k}\bigr]=\var_k,
\end{equation}
\begin{equation}
\label{eq:MTW2} w^{ij}\bigl[u_{ijkk} + c_{ijkk} + 2 \,c_{ijk,s}
T_{s,k} +c_{ij,s}T_{s,kk} + c_{ij,st} T_{s,k} T_{t,k} \bigr] \geq
\var_{kk}.
\end{equation}
We fix now $\bar x \in X$, we take $\eta$ a cut-off function
around $\bar x$, and define the function $G:X \times \SP^{n-1}
\to \R$,
$$
G(x,\xi):=\eta(x)^2\,w_{\xi \xi},\qquad w_{\xi \xi}:=\sum_{ij}
w_{ij}\xi^i \xi^j.
$$
We want to show that $G$ is uniformly bounded by a universal
constant $C$, depending only on $\dist(\bar x,\p X)$, $n$, the
cost function, and the function $h(x,p)$. (Observe that $G \geq
0$, since $(w_{ij})$ is positive definite.) In fact, this will
imply that
$$
\eta(x)^2\Bigl|D^2u(x) + D_{xx}
c\bigl(x,\cexp_x\bigl(\n u(x)\bigr)\bigr) \Bigr| \leq C,
$$
and since $\n u(x)$ is bounded and $c$ is smooth, the above
equation gives that $|D^2u|$ is locally uniformly bounded by a
universal constant, which is the desired a priori estimate.

To prove the bound on $G$, the strategy is the same of Theorem \ref{po1}: let $x_0
\in X$ and $\xi_0 \in \SP^{n-1}$ be a point where $G$ attains its
maximum. By a rotation of coordinates one can assume $\xi_0=e_1$.
Then at $x_0$ we have
\begin{equation}
\label{eq:MTW3} 0=(\log G)_i=\frac{w_{11,i}}{w_{11}} + 2\,
\frac{\eta_i}{\eta},
\end{equation}
$$
(\log G)_{ij}=\frac{w_{11,ij}}{w_{11}} + 2\,
\frac{\eta_{ij}}{\eta}- 6 \,\frac{\eta_i \eta_j}{\eta^2}.
$$
Since the above matrix is non-positive, we get
\begin{equation}
\label{eq:MTW4} 0 \geq w_{11} w^{ij}(\log G)_{ij}=w^{ij}w_{11,ij}
+ 2 \,\frac{w_{11}}{\eta}w^{ij}\eta_{ij}-
6\,w_{11}w^{ij}\,\frac{\eta_i \eta_j}{\eta^2}.
\end{equation}
We further observe that, differentiating the relation $\nabla u=-D_xc(x,T(x))$, we obtain the
relation
\begin{equation}
\label{eq:MTW5} w_{ij}=c_{i,k}T_{k,j}.
\end{equation}
This gives in particular $T_{k,j}=c^{k,i}w_{ij}$ (which
implies $|\n T| \leq C \,w_{11}$), and allows us to write derivatives
of $T$ in terms of that of $w$ and $c$.

The idea is now to start from \eqref{eq:MTW4}, and to combine the
information coming from \eqref{eq:MTW1}, \eqref{eq:MTW2},
\eqref{eq:MTW3}, \eqref{eq:MTW5}, to end up with a inequality of
the form
$$
0 \geq
w^{ij}\bigl[c^{k,\ell}c_{ij,k}c_{\ell,st}-c_{ij,st}\bigr]c^{s,p}c^{t,q}w_{p1}w_{q1}-
C,
$$
for some universal constant $C$. (When doing the computations,
one has to remember that the derivatives of $\var$ depend on
derivatives of $\n u$, or equivalently on derivatives of $T$.)
By a rotation of coordinates that leaves the \(e_1\) direction invariant,\footnote{This is possible since \(e_1\) is a maximum point for the map \(\xi\mapsto \eta(x_0)w_{\xi\xi}\),
therefore (except  in the trivial case \(\eta(x_0)=0\)) \(w_{1i}=0\) for every \(i\ne 1\).} one can further assume that
$(w_{ij})$ is diagonal at $x_0$. We then obtain
$$
w^{ii}\bigl[c^{k,\ell}c_{ii,k}c_{\ell,st}-c_{ii,st}\bigr]c^{s,1}c^{t,1}w_{11}w_{11}
\leq C.
$$
Up to now, the $\MTW$ condition has not been used. 

We now
apply $\MTW(K)$ to the vectors $\xi_1=(0,\sqrt{w^{22}}, \ldots,\sqrt{w^{nn}})$
and $\xi_2=(w_{11},0, \cdots,0)$ 
 to get
 $$
  Kw_{11}^2 \sum_{i=2}^n w^{ii}
   \leq
 C+w^{11}\bigl[c^{k,\ell}c_{11,k}c_{\ell,st}-c_{11,st}\bigr]c^{s,1}c^{t,1}w_{11}w_{11},
$$
which gives (using {\bf (C0)} and the fact that $w^{11}=(w_{11})^{-1}$)
\begin{equation}\label{eq:MTW final}
w_{11}^2 \sum_{i=2}^n w^{ii} \leq C\bigl(1+w_{11}\bigr) .
\end{equation}
Recalling that $w^{ij}=(w_{ij})^{-1}$, by the arithmetic-geometric inequality and
\eqref{eq:jacob w} we have
$$
\frac{1}{n-1}\sum_{i=2}^n w^{ii} \geq \Bigl(
\prod_{i=2}^n w^{ii} \Bigr)^{1/(n-1)} \geq c_0\,(w^{11})^{-1/(n-1)}=c_0\,w_{11}^{1/(n-1)},
$$
where $c_0:=\inf_{x \in X} h(x,\n u(x))^{-1/(n-1)} >0$. Hence,
combining the above estimate with \eqref{eq:MTW final} we finally
obtain
$$
\bigl[w_{11}(x_0)\bigr]^{2+1/(n-1)} \leq C\bigl(1+w_{11}(x_0)\bigr) ,
$$
which proves that $G(x,\xi) \leq G(x_0,\xi_0) \leq C$ for all
$(x,\xi) \in X \times \SP^{n-1}$, as desired.
\end{proof}

\begin{remark}
{\rm
Apart from having its own interest, this general theory of Monge-Amp\`ere type equations
satisfying the MTW condition turns out to be useful also in the classical case.

Indeed the MTW condition can be proven to be coordinate invariant \cite{MTW,Loe1,KMC1}.
This implies that, if $u$ solves an equation of the form \eqref{eq:MAgeneral}
with $\mathcal A$ satisfying the MTW condition (see \eqref{eq:MTW tensor}) and $\Phi$ is a smooth diffeomorphism, then $u\circ \Phi$ satisfies an equation of the same form with a new matrix
$\tilde{\mathcal A}$ which still satisfies the MTW condition.
This means for instance that, if one wants to prove some boundary regularity result, through
a smooth diffeomorphism one can reduce himself to a nicer situation.
In particular this remark applies to the classical Monge-Amp\`ere equation
(which trivially satisfies MTW(0)):
if $u$ solves \eqref{eq:MAclassical} then $u\circ \Phi$ satisfies \eqref{eq:MAgeneral}
for some matrix $\mathcal A$ for which MTW(0) holds.

This observation is extremely natural in relation to the optimal transport problem:
Consider for instance the optimal transport problem with cost $-x\cdot y$ between two smooth densities 
supported on smooth uniformly convex domains. By Theorem \ref{thm:regularity transport}
we know that the optimal map $T$ is smooth. However, if one performs smooth change of coordinates
$\Phi:X\to \Phi(X)$ and $\Psi:Y\to \Psi(Y)$, under these transformations the cost function becomes $-\Phi^{-1}(x)\cdot \Psi^{-1}(y)$ and the optimal transport map is given by $\Psi\circ T\circ \Phi^{-1}$, which is still smooth.
However this regularity of the optimal map cannot be deduced anymore from Theorem \ref{thm:regularity transport}, while Theorem \ref{thm:mtw} still 
applies providing the expected regularity.
}
\end{remark}

\medskip    \subsection{A geometric interpretation of the MTW condition}

Although the $\MTW$ condition seemed the right assumption to
obtain regularity of optimal maps, it was only after Loeper's work
\cite{Loe1} that people started to have a good understanding of
this condition, and a more geometric insight. The idea of Loeper
was the following: for the classical Monge-Amp\`ere equation, a
key property to prove regularity of convex solutions is that the
subdifferential of a convex function is convex, and so in
particular connected. Roughly speaking, this has the following
consequence: whenever a convex function $u$ is not $C^1$ at a
point $x_0$, there is at least a whole segment contained in the
subdifferential of $u$ at $x_0$, and this fact combined with
the Monge-Amp\`ere equation provides a contradiction. (See also
Theorem \ref{thm:C1a} below.) Hence, Loeper wanted to understand
whether the $c$-subdifferential of a $c$-convex function is at
least connected (recall \eqref{eq:csubdiff}), believing that this fact had a link with the
regularity.\\

We wish to find some simple enough conditions
implying the connectedness of sets $\p_c\psi$. 
Of course if $\p_c\psi(\bar x)$ is a singleton the connectedness is automatic, 
so we should look for non-smooth functions.

The easiest example is provided by the maximum of two cost functions:
fix $y_0,y_1 \in Y$, $a_0,a_1 \in \R$, and define
$$
\psi(x):=\max\bigl\{-c(x,y_0) + a_0,-c(x,y_1) + a_1\bigr\}.
$$
Take a point $\bar x \in \{x \, | \,-c(x,y_0) + a_0=-c(x,y_1) +
a_1\}$, and let $\bar y \in \p_c\psi(\bar x)$. Since
$\psi(x)+c(x,\bar y)$ attains its minimum at $x=\bar x$, we get (recall \eqref{Frechet subdiff}
for the definition of the Fr\'echet subdifferential)
$$
0 \in \p_{\bar x}^-\bigl(\psi + c(\cdot,\bar y)\bigr),
$$
or equivalently
$$
-\n_x c(\bar x,\bar y) \in \p^-\psi(\bar x).
$$
From the above inclusion, one can easily deduce that $\bar y \in
\cexp_{\bar x}\bigl(\p^-\psi(\bar x)\bigr)$. Moreover, it is not
difficult to see that
$$
\p^-\psi(\bar x)=\{ (1-t)v_0 + t v_1\,|\,t \in [0,1]\}, \qquad
v_i:=\n_x c(\bar x,y_i)=(\cexp_{\bar x})^{-1}(y_i), \quad i=0,1.
$$
Therefore, denoting by $[v_0,v_1]$ the segment joining $v_0$ and
$v_1$, we obtain
$$
\p_c\psi(\bar x)\subset \cexp_{\bar x}\bigl([v_0,v_1]\bigr).
$$

The above formula suggests the following definition:
\begin{definition}
Let $\bar x \in X$, $y_0,y_1 \in Y$. Then we
define the \emph{$c$-segment} from $y_0$ to $y_1$ with base $\bar
x$ as
$$
[y_0,y_1]_{\bar x}:=\bigl\{ y_t=\cexp_{\bar x}\bigl((1-t)
(\cexp_{\bar x})^{-1}(y_0) + t (\cexp_{\bar x})^{-1}(y_1)\bigr) \,
|\, t \in [0,1]\bigr\}.
$$
\end{definition}

In \cite{MTW}, Ma,
Trudinger, and Wang showed that, in analogy with the quadratic case (see Section \ref{sect:Bre Alex}), the convexity of $D_xc(x,Y)$
for all $x \in X$ is necessary for regularity. 
By slightly modifying their argument, Loeper \cite{Loe1}
showed that the connectedness of the
$c$-subdifferential is a necessary condition for the smoothness of
optimal transport (see also \cite[Theorem 12.7]{Vil}):

\begin{theorem}
\label{thm:no connect counterexample} Assume that there exist
$\bar x \in X$ and $\psi:X \to \R$ $c$-convex such that
$\p_c\psi(\bar x)$ is not  connected. Then one can
construct smooth positive probability densities $f$ and $g$ such that the optimal map is
discontinuous.
\end{theorem}

While the above result was essentially contained in \cite{MTW},
Loeper's major contribution was to link the connectedness of the
$c$-subdifferential to a differential condition on the cost
function, which actually coincides with the $\MTW$ condition (see
Section \ref{subsect:MTW}). In \cite{Loe1} he proved (a slightly weaker version of) the following result (see \cite[Chapter 12]{Vil} for
a more general statement):
\begin{theorem}
\label{thm:regular cost} The following conditions are equivalent:
\begin{enumerate}
\item[(i)] For any $\psi$ $c$-convex, for all $\bar x \in X$,
$\p_c\psi(\bar x)$ is connected. \item[(ii)] For any $\psi$
$c$-convex, for all $\bar x \in X$, $(\cexp_{\bar
x})^{-1}\bigl(\p_c\psi(\bar x)\bigr)$ is convex, and it coincides
with $\p^-\psi(\bar x)$. \item[(iii)] For all $\bar x \in X$, for
all $y_0,y_1 \in Y$, if $[y_0,y_1]_{\bar x}=(y_t)_{t \in [0,1]}\subset Y$ then
\begin{equation}
\label{eq:mtw inequality} c(x,y_t) - c(\bar x,y_t) \geq
\min\bigl[ c(x,y_0) - c(\bar x,y_0), c(x,y_1) - c(\bar
x,y_1)\bigr]
\end{equation}
for all $x \in X$, $t \in [0,1]$. \item[(iv)] For all $\bar x
\in X$, $y \in Y$, $\eta,\xi \in \R^n$ with $\xi \perp \eta$,
$$
\left.\frac{d^2}{ds^2}\right|_{s=0}
\left.\frac{d^2}{dt^2}\right|_{t=0}\, c\bigl(\cexp_{\bar
x}(t\xi),\cexp_{\bar x}(p+s\eta) \bigr) \leq 0,
$$
where $p =(\cexp_{\bar x})^{-1}(y)$.
\end{enumerate}
Moreover, if any of these conditions is not satisfied, $C^1$
$c$-convex functions are not dense in the set of Lipschitz $c$-convex
functions.
\end{theorem}

\begin{proof}[Sketch of the proof]
We give here only some elements of the proof.\\

\noindent \textit{(ii) $\Rightarrow$ (i).} Since $(\cexp_{\bar
x})^{-1}\bigl(\p_c\psi(\bar x)\bigr)$ is convex,
it is connected, and so its image by $\cexp_{\bar x}$ is connected too.\\

\noindent \textit{(i) $\Rightarrow$ (ii).} For $\psi_{\bar
x,y_0,y_1}:=\max\bigl\{-c(\cdot,y_0) + c(\bar x,y_0),-c(\cdot,y_1)
+ c(\bar x,y_1)\bigr\}$ we have
$$
(\cexp_{\bar
x})^{-1}\bigl(\p_c\psi_{\bar x,y_0,y_1}(\bar x)\bigr) \subset
[(\cexp_{\bar x})^{-1}(y_0),(\cexp_{\bar x})^{-1}(y_1)],
$$
 which is a
segment. Since in this case connectedness is equivalent to
convexity, if (i) holds we obtain $\p_c\psi_{\bar x,y_0,y_1}(\bar
x)=[y_0,y_1]_{\bar x}=\cexp_{\bar x}\bigl(\p^-\psi_{\bar x,y_0,y_1}(\bar x) \bigr)$.

In the general case, we fix $y_0,y_1 \in \p_c\psi(\bar x)$. Then
it is simple to see that
$$
\p_c\psi(\bar x) \supset \p_c\psi_{\bar x,y_0,y_1}(\bar
x)=[y_0,y_1]_{\bar x},
$$
and the result follows easily.\\

\noindent \textit{(ii) $\Leftrightarrow$ (iii).} Condition
\eqref{eq:mtw inequality} is equivalent to $\p_c\psi_{\bar
x,y_0,y_1}=[y_0,y_1]_{\bar x}$.
Then the equivalence between (ii) and (iii) follows arguing as above.\\

\noindent \textit{(iii) $\Rightarrow$ (iv).} Fix $\bar x \in X$ and $y \in Y$
with $y=\cexp_{\bar x}(p)$. Take $\xi,\eta$ orthogonal and with
unit norm, and define
$$
y_0:=\cexp_{\bar x}(p-\e \eta),\quad y_1:=\cexp_{\bar x}(p+\e
\eta)\qquad\text{for some $\e>0$ small}.
$$
Moreover, let
$$
h_0(x):=c(\bar x,y_0)-c(x,y_0),\quad h_1(x):=c(\bar
x,y_1)-c(x,y_1),\quad \psi:=\max\bigl\{h_0,h_1\bigr\}=\psi_{\bar
x,y_0,y_1}.
$$
We now define a curve $\g(t)$  contained in the set
$\{h_0=h_1\}$ such that $\g(0)=\bar x$, $\dot\g(0)=\xi$.

Since $y \in [y_0,y_1]_{\bar x}$, by (iii) we get $y \in
\p_c\psi(\bar x)$, so 
\begin{align*}
\frac{1}{2}\bigl[h_0(\bar x)+h_1(\bar x) \bigr]+c(\bar
x,y)&=\psi(\bar x)+c(\bar x,y)\leq
\psi(\g(t))+c(\g(t),y)\\&=\frac{1}{2}\bigl[h_0(\g(t))+h_1(\g(t))
\bigr]+c(\g(t),y),
\end{align*}
where we used that $h_0=h_1$ along $\g$. Recalling the definition
of $h_0$ and $h_1$ we deduce that
$$
\frac{1}{2}\bigl[c(\g(t),y_0)+c(\g(t),y_1) \bigr]-c(\g(t),y) \leq
\frac{1}{2}\bigl[c(\bar x,y_0)+c(\bar x,y_1) \bigr]-c(\bar x,y),
$$
so the function $t \mapsto
\frac{1}{2}\bigl[c(\g(t),y_0)+c(\g(t),y_1) \bigr]-c(\g(t),y)$
achieves its maximum at $t=0$. This implies
$$
\left.\frac{d^2}{dt^2}\right|_{t=0}
\Bigl[\frac{1}{2}\bigl(c(\g(t),y_0)+c(\g(t),y_1)
\bigr)-c(\g(t),y)\Bigr] \leq 0,
$$
i.e.,
$$
\Bigl\<\Bigl[\frac{1}{2}\bigl(D_{xx}c(\bar x,y_0)+D_{xx}c(\bar
x,y_1) \bigr) -D_x^2c(\bar x,y)\Bigr]\cdot \xi,\xi\Bigr\> \leq 0
$$
(here we used that $\n_x c(\bar x,y)=\frac{1}{2}\bigl[\n_x c(\bar
x,y_0)+\n_x c(\bar x,y_1) \bigr]$). Thus the function
$$
s \mapsto \bigl\<D_{xx}c\bigl(\bar x,\cexp_{\bar
x}(p+s\eta)\bigr)\cdot \xi,\xi\bigr\>
$$
is concave, proving (iv).
\end{proof}

To understand why the above properties are related to smoothness,
consider Theorem \ref{thm:regular cost}(iii). It says that if we
take the function $\psi_{\bar x,y_0,y_1}=\max\bigl\{-c(\cdot,y_0)
+ c(\bar x,y_0),-c(\cdot,y_1) + c(\bar x,y_1)\bigr\}$, then we are
able to touch the graph of this function from below at $\bar x$
with the family of functions $\{-c(\cdot,y_t) + c(\bar x,y_t)\}_{t
\in [0,1]}$. This suggests that we could use this family to
regularize the cusp of $\psi_{\bar x,y_0,y_1}$ at the point $\bar
x$, by slightly moving above the graphs of the functions
$-c(\cdot,y_t) + c(\bar x,y_t)$. On the other hand, if
\eqref{eq:mtw inequality} does not hold, it is not clear how to
regularize the cusp preserving the condition of being $c$-convex.

By what we said above, the regularity property seems mandatory to
develop a theory of smoothness of optimal transport. Indeed, if it
is not satisfied, we can construct $C^\i$ strictly positive
densities $f,g$ such that the optimal map is not continuous. Hence
the natural question is when it is satisfied, and what is the link
with the $\MTW$ condition.

Note also that the failure of the connectedness of \(\partial_c u\) is a local obstruction to the regularity of optimal transport maps to be compared to the case of the quadratic cost where the obstruction to the regularity is the failure of the convexity of \(Y\), a global condition. 

\medskip    \subsection{Regularity results}
\label{subsect:MTW} As we have seen in Theorem \ref{thm:regular
cost}, roughly speaking 
the connectedness of the $c$-subdifferential of $c$-convex functions is equivalent to
\begin{equation}
\label{eq:4 deriv d2} \left.\frac{d^2}{ds^2}\right|_{s=0}
\left.\frac{d^2}{dt^2}\right|_{t=0}\,
c\bigl(\cexp_x(t\xi),\cexp_x(p+s \eta) \bigr) \leq 0,
\end{equation}
for all $p,\xi,\eta\in \R^n$, with $\xi$ and $\eta$
orthogonal, $p=(\cexp_x)^{-1}(y)$.
(This equivalence is not exact, since to show that \eqref{eq:4 deriv d2}
implies the connectedness of the $c$-subdifferential of $c$-convex functions  one needs to know that
the target is $c$-convex to ensure that it contains the $c$-segment $[y_0,y_1]_{\bar x}$, see (iii) in Theorem \ref{thm:regular cost} above.)

By some standard but tedious computations it is not
difficult to check 
that the above expression \emph{coincides} up to the sign with the
$\MTW$ tensor. Hence
$$
\MTW(0)\qquad \Leftrightarrow \qquad \eqref{eq:4 deriv d2}
\qquad \Leftrightarrow \qquad \p_c\psi(x) \text{ connected }\ \ \forall\,x,\,\psi \text{ $c$-convex},
$$
and by Theorems \ref{thm:mtw} and \ref{thm:no connect
counterexample} the MTW condition is necessary and sufficient for
the smoothness of the optimal transport map.\\

By exploiting (a variant of) Theorem \ref{thm:regular cost},
Loeper \cite{Loe1} proved the following regularity result (recall the notation in Theorem \ref{cbrenier}):
\begin{theorem}
\label{thm:C1a}  
Let $c:X\times Y \to \R$ satisfy {\bf (C0)}-{\bf (C3)}. 
Assume that the $\MTW(K)$ holds for some $K>0$, $f$ is bounded
from above on $X$, $g$ is bounded away from zero on $Y$,
and denote by $T$ the optimal transport map sending $f$ onto $g$.
Finally, suppose that $D_xc(x,Y)$
is convex for any $x \in X$. Then $u \in
C^{1,\a}(\overline{X})$, with $\a=1/(4n-1)$ (hence $T_u \in
C^{0,\a}(\overline{X})$).
\end{theorem}

In a subsequent paper, Liu \cite{Liu} proved the above result with
 $\a=1/(2n-1)$ and showed that such exponent is optimal.

We notice however that this result does not include Caffarelli's result
since the cost $-x\cdot y$ satisfies $\MTW(0)$. Hence, it would be nice to have 
a H\"older regularity theorem which includes at the same time Caffarelli's and Loeper's results.

As proved by the second author together with  Kim and McCann \cite{FKM}, such a result is indeed true
(see also \cite{FL,GK, Wan2}):

\begin{theorem}
\label{thm:C1a mtw0}  
Let $c:X\times Y \to \R$ satisfy {\bf (C0)}-{\bf (C3)}. 
Assume that the $\MTW(0)$ holds, $f$ is bounded
from above on $X$, and $g$ is bounded away both from zero and infinity on $Y$.
Also, suppose that $D_xc(x,Y)$ and $D_yc(Y,y)$
are uniformly convex for any $x \in X$ and $y \in Y$. 
Then $u \in
C^{1,\a}_{\rm loc}(\overline{X'})$ for any open set $X'\subset X$ where $f$
is uniformly bounded away from zero.
\end{theorem}

Let us mention that further extensions of Theorem \ref{thm:regularity transport} to this general setting have been obtained in \cite{LTW, LTWW2p,DF2}.

Note that Theorem \ref{thm:C1a mtw0} as well as the other results in \cite{LTW, LTWW2p,DF2,GK, Wan2} only deal
with the interior regularity for optimal transport maps.
It would be interesting to develop a boundary regularity theory, in the spirit of \cite{Caf bdry,Caf4},
for the class of equations
\eqref{eq:MAgeneral} arising in optimal transport.\\

We conclude this section providing some examples of cost functions satisfying the MTW condition (where defined) and referring to \cite{MTW} for more example.
\begin{itemize}
\item[-] \(c_1(x,y)=-\log|x-y|\);
\item[-] \(c_2(x,y)=\sqrt{a^2-|x-y|^2}\);
\item[-] \(c_3(x,y)=\sqrt{a^2+|x-y|^2}\);
\item[-] \(c_4(x,y)=|x-y|^p\) with  \(-2<p<1\).
\end{itemize}
Note that the last two costs provide (respectively as \(a\to 0^+\) and \(p\to 1^-\)) an approximation of the Monge cost \(c(x,y)=|x-y|\). In spite of the fact that this latter cost does not satisfy any of the conditions \(\mathbf{(C0)}\)-\(\mathbf{(C3)}\), still existence of optimal transport maps can be proved
(see for instance \cite{AmbrosioTrasporto} for an account of the theory)
and some regularity results have been obtained in \cite{Pratellieco}. Very recently, the strategy to approximate the Monge cost with \(c_3\) has been used in \cite{SanWan} to deduce some estimates for the Monge problem by proving a priori bounds on the transport maps which are uniform as $a \to 0$.
\medskip    \subsection{The case of Riemannian manifolds}
\label{sect:Riem}
Let us consider the case when $c=d^2/2$ on $X=Y=(M,g)$
a compact Riemannian manifold without boundary. 
As we have seen in the previous section, the MTW condition and some suitable convexity of the domains
are necessary to obtain regularity of optimal maps. Here we have also to face an additional problem:
indeed, while before the cost function was assumed 
to be smooth on $X\times Y$, in the case of a compact Riemannian manifold without boundary the function $d^2/2$ is never smooth on the whole $M\times M$ due to the presence of the cut-locus
 (see Section \ref{sect:cut} below for a more precise description).

Before focusing on all these issues, let us start with the following  example which will allow us to look at the MTW condition from a different perspective.

\begin{example}
We want to show how negative
sectional curvature is an obstruction to regularity (indeed even
to continuity) of optimal maps when $c=d^2/2$.
(We refer to \cite[Theorem
12.4]{Vil} for more details on the construction given below.)

Let $M=\H^2$ be the hyperbolic plane (or a compact quotient
thereof). Fix a point $O$ as the origin, and fix a local system of
coordinates in a neighborhood of $O$ such that the maps
$(x_1,x_2)\mapsto (\pm x_1,\pm x_2)$ are local isometries (it
suffices for instance to consider the model of the Poincar\'e
disk, with $O$ equal to the origin in $\R^2$). Then define the
points
$$
A^\pm=(0,\pm \e), \quad B^\pm=(\pm \e,0)\qquad\text{for some
$\e>0$}.
$$
Take a measure $\mu$ symmetric with respect to $0$ and
concentrated near $\{A^+\} \cup \{A^-\}$ (say $3/4$ of the total
mass belongs to a small neighborhood of $\{A^+\} \cup \{A^-\}$),
and a measure $\nu$ symmetric with respect to $0$ and concentrated
near $\{B^+\} \cup \{B^-\}$. Moreover assume that $\mu$ and $\nu$
are absolutely continuous, and have positive smooth densities.
We denote by $T$ the unique optimal transport map, and
we assume by contradiction that $T$ is continuous. By symmetry we
deduce that $T(O)=O$, and by counting the total mass there
exists a point $A'$ close to $A^+$ which is sent to a point $B'$
near, say, $B^+$.

But, by negative curvature (if $A'$ and $B'$ are close enough to
$A$ and $B$ respectively), Pythagoras Theorem becomes an
inequality: $d(O,A')^2 + d(O,B')^2<d(A',B')^2$, and this
contradicts the optimality of the transport map, as transporting
$A'$ onto $O$ and $O$ onto $B'$ would be more convenient than
transporting $A'$ onto $B'$ and letting $O$ stay at rest.
\end{example}

Now, the natural question is: how does the above example fit into
Ma, Trudinger and Wang and Loeper's results? The answer was
given by Loeper \cite{Loe1} where he noticed that, when $c=d^2/2$
on a Riemannian manifold $M$, the
$\MTW$ tensor satisfies the following remarkable identity:
\begin{equation} \label{eq:mtw sect}
\Smtw_{(x,x)}(\xi,\eta)=\frac 2 3 {\rm
Sect}_x([\xi,\eta]),
\end{equation}
where $\xi,\eta \in T_xM$ are two orthogonal unit vectors, and
${\rm Sect}_x([\xi,\eta])$ denotes the sectional curvature of the
plane generated by $\xi$ and $\eta$.

Combining \eqref{eq:mtw sect} with
Theorems \ref{thm:no connect counterexample} and \ref{thm:regular
cost}, we get the following important negative result:
\begin{theorem}
\label{thm:negativ curv counterexample} Let $(M,g)$ be a (compact)
Riemannian manifold, and assume there exist $x \in M$ and a
plane $P \subset T_xM$ such that ${\rm Sect}_x(P)<0$.  Then there
exist smooth positive probability densities $f$ and $g$
such that the optimal map for the cost $c=d^2/2$ is discontinuous.
\end{theorem}

After this negative result, one could still hope to develop a
regularity theory on any manifold with nonnegative sectional
curvature. But this is not the case: as shown by Kim \cite{Kim}
(see also \cite{FRVsurfaces})
the regularity condition is strictly stronger than the condition
of nonnegativity of sectional curvatures. In conclusion, except
for some special cases, the optimal map is
non-smooth.

Some examples of manifolds satisfying the $\MTW$ condition have been found in \cite{Loe1,Loe2,KMC1,KMC2,FR,FRV}:
\begin{enumerate}
\item[-] $\R^n$ and $\T^n$ satisfy $\MTW(0)$.
\item[-] $\SP^n$, its quotients (like $\R\mathbb{P}^n$),
and its submersions (like $\C\mathbb{P}^n$ or $\mathbb H \mathbb
P^n$), satisfy $\MTW(1)$.
\item[-] Products of any of the
examples listed above (for instance, $\SP^{n_1} \times \ldots
\times \SP^{n_k} \times \R^\ell$ or $\SP^{n_1} \times
\C\mathbb{P}^{n_2} \times \T^{n_3}$) satisfy $\MTW(0)$.
\item[-] smooth perturbations of $\SP^n$ satisfy $\MTW(K)$ for some $K>0$.\\
\end{enumerate}

Now, in order to prove regularity of optimal transport maps on manifolds satisfying MTW one would like 
to localize 
 Theorems \ref{thm:C1a} and \ref{thm:C1a mtw0}.
 However, as we already mentioned above, one has to face a very nontrivial issue:
 the cost function $d^2/2$ is not everywhere smooth due to the presence of the cut-locus, hence to localize that 
 theorems one would like to know the validity of the so-called ``stay-away property'':
 \begin{equation}
 \label{eq:stayaway}
 T(x)\not\in \cut(x) \qquad \forall\,x \in M.
 \end{equation}
 (See the next section for a precise definition of $\cut(x)$.)
 While this property 
 has been proven to hold true in some special cases \cite{DelLoe,DG,FKM-spheres},
 it is still unknown in general and this creates several difficulties in the proof of regularity of optimal maps.
 In particular this is one of the reasons why, on perturbations of $\SP^n$,
 only continuity (and not higher regularity) of optimal maps is currently known \cite{FRV}.

Another issue is the convexity of the target: in the case of compact manifolds without boundary, the assumption that 
$-D_xc(x,Y)$
is convex for every $x \in X$ (see for instance Theorem \ref{thm:C1a}) corresponds
to the convexity of the cut-locus of the manifold when seen from any tangent space.
This shows that regularity of solutions is strongly related to the convexity of the cut-locus.
Actually, as we shall describe now, the MTW condition has even much stronger links with the cut-locus.


\medskip

\subsection{MTW v.s. cut-locus}
\label{sect:cut}
To explain the connection between the MTW condition and the convexity of the cut-locus, we start by recalling some basic definitions.

\label{subsect:cut} Given a tangent vector $v \in
T_xM$, the curve $\left(\exp_x(tv)\right)_{t\geq 0}$ is a geodesic
defined for all times, but in general is not minimizing for large
times. On the other hand, it is possible to prove that
$\exp_x(tv)$ is always minimizing between $x$ and $\exp_x(\e v)$
for $\e>0$ sufficiently small. We define the \textit{cut-time}
$t_C(x,v)$ as
$$
t_C(x,v):=\inf \Bigl\{t>0 \, : \, s \mapsto \exp_x(sv) \text{ is
not minimizing between $x$ and $\exp_x(t v)$} \Bigr\}.
$$
Given $x \in M$, we define the \textit{cut-locus} of $x$
$$
\cut(x):=\Bigl\{ \exp_x\bigl(t_C(x,v)v\bigr)\, :\, v \in T_xM,\,
|v|_x=1\Bigr\},
$$
the \textit{tangent cut-locus} of $x$
$$
\TCL(x) := \bigl\{ t_C(x,v)v\, : \, v \in T_xM,\, |v|_x=1\bigr\},
$$
and the \textit{injectivity domain} of the exponential map at $x$
$$
\I(x) := \bigl\{ tv\, : \, 0\leq t< t_C(x,v),\ v\in T_xM,\,
|v|_x=1\bigr\}.
$$
With these definitions, we have
$$
\cut(x) = \exp_x \bigl(\TCL(x)\bigr),\qquad \TCL(x)=\partial
\bigl(\I(x)\bigr).
$$
For instance, on  the sphere $\SP^n$, $t_C(x,v)=\pi/|v|_x$, \(\cut(x)=\{-x\}\),  \(\I(x)= B_\pi(0)\), and  \(\TCL(x)=\partial B_\pi(0)\).

It is possible to prove that, if $y \not\in \cut(x)$, then $x$ and
$y$ are joined by a unique minimizing geodesic. The converse is
close to be true: $y\not\in\cut(x)$ if and only if there are
neighborhoods $U$ of $x$ and $V$ of $y$ such that any two points
$x'\in U$ and $y'\in V$ are joined by a unique minimizing geodesic.
In particular $y\not\in\cut(x)$ if and only if
$x\not\in\cut(y)$.\\

%
%

The fact that a point $y \in M$ belongs to $\cut(x)$ is a
phenomenon which is captured by the regularity of the distance
function. Indeed, it can be proven that the following holds (see
for instance \cite[Proposition 2.5]{CEMCS}):
\begin{enumerate}
\item[(a)] The function $d(x,\cdot)^2$ is smooth
in a neighborhood of $y$ if and only if $y \not\in \cut(x)$.
\item[(b)] The function $d(x,\cdot)^2$ has an upward Lipschitz cusp at $y$
if and only if $y \in \cut(x)$ and there are at least two
minimizing geodesics between $x$ and $y$.
\item[(c)] The
function $d(x,\cdot)^2$ is $C^1$ at $y$ and its Hessian has an
eigenvalue $-\infty$ if and only if $y \in \cut(x)$ and there is a
unique minimizing geodesics between $x$ and $y$. 
\end{enumerate}
In the above statement, having an ``upward cusp'' means that there
exist two vectors $p_1\neq p_2$ both belonging to the
supergradient of $f:=d(x,\cdot)^2$ at $y$: writing everything in
charts, we have
$$
\{p_1,p_2 \} \subset \p^+f(y):=\bigl\{p \,:\, f(y+ v) \leq f(y) +
\<p, v\> +o(|v|) \quad \forall\,v\bigr\},
$$
that is $f$ is locally below the function $v \mapsto
f(y)+\min\{\<p_1, v\>,\<p_2, v\>\} + o(|v|)$ near $y$. Hence (b) 
corresponds to roughly say that the second derivative (along the
direction $p_2-p_1$) of $d(x,\cdot)^2$ at $y$ is $-\infty$. (The
fact that there is an upward cusp means that one of the second
directional derivatives is a negative delta measure.)

Furthermore, saying that ``Hessian has an eigenvalue $-\infty$''
means that (always working in charts)
$$
\liminf_{|v| \to 0} \frac{f(y+v) - 2 f(y) +f(y-v)}{|v|^2}
=-\infty.
$$
Thus, the above description of the cut-locus in terms of the
squared distance can be informally summarized as follows:
\begin{equation}
\label{eq:belong cut}
 y \in \cut(x)
\quad \Leftrightarrow \quad D_{yy} d^2(x,y)v\cdot
v=-\infty \quad \text{for some }v \in T_yM\,.
\end{equation}
\\

Based on this observation, in \cite{LV} Loeper and
Villani noticed the existence of a deep connection between the
$\MTW$ condition and the geometry of the cut-locus.
The idea is
the following: fix $x \in M$, and let $v_0,v_1 \in \I(x)$.
Consider the segment $(v_t)_{t \in [0,1]}$, with $v_t:=(1-t) v_0
+t v_1$. Set further $y_t :=\exp_x(v_t)$. Since $v_0,v_1 \in
\I(x)$ we have
$$
y_0,y_1 \not \in \cut(x).
$$
In particular $c(x,\cdot):=d(x,\cdot)^2/2$ is smooth in a
neighborhood of $y_0$ and $y_1$. Assume now that $\MTW(0)$
holds. Thanks to Theorem \ref{thm:regular cost}(iv), we
know that the function
$$
\eta \mapsto \bigl\<D_{xx}c\bigl(\bar x,\exp_{\bar
x}(p+\eta)\bigr)\cdot \xi,\xi\bigr\>
$$
is concave for all $\eta \perp \xi$. (This is just a formal
argument, as the theorem applies a priori only if $\exp_{\bar
x}(p+\eta) \not \in \cut(\bar x)$.) Applying this fact along the
segment $(v_t)_{t \in [0,1]}$ and exploiting the smoothness of
$d(x,\cdot)^2$ near $y_0$ and $y_1$, for $\xi \perp
(v_1-v_0)$  we obtain
$$
\inf_{t \in [0,1]} \bigl\<D_{xx} d^2(x, y_t)\cdot \xi,\xi\bigr\>
\geq \min\bigl\{\bigl\<D_{xx} d^2(x, y_0)\cdot
\xi,\xi\bigr\>,\,\bigl\<D_{xx} d^2(x, y_1)\cdot \xi,\xi\bigr\>
\bigr\} \geq C_0,
$$
for some constant $C_0 \in \R$. Hence, if we forget for a moment
about the orthogonality assumption between $v_1-v_0$ and $\xi$, we
see that the above equation implies that $x \not \in \cut(y_t)$
for all $t \in [0,1]$ (compare with \eqref{eq:belong cut}), which
by symmetry gives
$$
y_t \not \in \cut(x) \qquad \forall \, t \in [0,1],
$$
or equivalently
$$
v_t \not \in \TCL(x) \qquad \forall \, t \in [0,1].
$$
Since $v_0,v_1 \in \I(x)$, we have obtained
$$
v_t \in \I(x) \qquad \forall \, t \in [0,1],
$$
that is $\I(x)$ is convex. In conclusion, this formal argument
suggests that the $\MTW$ condition (or a variant of it) should
imply that all tangent injectivity domains $\I(x)$ are convex, for
every $x \in M$.\\

This would be a remarkable property. Indeed,
usually the only regularity results available for $\I(x)$ say that
$\TCL(x)$ is just Lipschitz \cite{IT}. Moreover, such a
result would be of a global nature, and not just local like a
semi-convexity property.
Unfortunately the argument described above is just formal and up
to now there is no complete result in that direction. Still, one
can prove some rigorous results in some special cases \cite{LV,FR,FRVsurfaces,FRV,FRVnec}.

As we already mentioned before,
the convexity of the tangent cut loci is not only an interesting geometric property but 
it is also important for the regularity of optimal transport maps,
since it corresponds to the assumption that $-D_xc(x,Y)$
is convex for every $x \in X$. 
Indeed, as shown in \cite{FRVnec}, both the MTW condition
and the convexity of the tangent cut-loci are necessary (and in some cases
sufficient) conditions for the 
continuity of optimal transport maps on manifolds.

\medskip    \subsection{Partial regularity}\label{secPR}

In the case of the classical Monge-Amp\`ere equation we saw that convexity 
of the target is a necessary condition for the regularity of Brenier solutions to the Monge-Amp\`ere equation.
Now, in the case of a general cost, in addition to some suitable
convexity of the target one needs to assume the validity of the MTW condition.
This condition is however known to be true only for few examples of costs (for instance, all costs  of the form \(|x-y|^p\) with \(p\notin(-2,1)\cup \{2\}\) do not satisfy the MTW condition, see however \cite{CDMN} for some ``perturbative'' result in this case). Moreover only few manifolds are known to satisfy it (essentially just the ones listed in Section \ref{sect:Riem}), and for instance it is known to  fail on sufficiently  flat ellipsoids \cite{FRVsurfaces}. \\

Let us also  recall that, in the case of a Riemannian manifold, the MTW condition
implies that the sectional curvature is nonnegative at every point (see \eqref{eq:mtw sect}).
Therefore, if we consider a compact manifold $(M,g)$ with negative sectional curvature
we have that the MTW condition fails \emph{at every point}.
This fact could make one suspect that the transport map could be extremely irregular.
However, as shown by the authors
in  \cite{DFpartial}, this negative picture cannot happen:

\begin{theorem}
\label{thm:c}
Let $X,Y\subset \R^n$ be two bounded open sets, and let
$f:X\to \R^+$ and $g:Y\to \R^+$ be two
continuous
probability densities, respectively 
bounded away from zero and infinity on $X$ and $Y$.
Assume that the cost $c:X\times Y \to \R$ satisfies {\bf (C0)}-{\bf (C3)}, and denote by $T:X\to Y$ the unique optimal transport map sending $f$ onto $g$.
Then there exist two relatively closed sets $\Sigma_X \subset X,\Sigma_Y\subset Y$ of measure zero such that
$T:X\setminus \Sigma_X \to Y\setminus \Sigma_Y$ is a homeomorphism of class $C_{\rm loc}^{0,\beta}$ for any $\beta<1$.
In addition, if $c\in C^{k+2,\alpha}_{\rm loc}(X\times Y)$, $f\in C^{k,\alpha}_{\rm loc}(X)$, and $g\in C^{k,\alpha}_{\rm loc}(Y)$
for some $k \geq 0$ and $\alpha \in (0,1)$, then $T:X\setminus \Sigma_X\to Y\setminus \Sigma_Y$ is a
diffeomorphism of class $C_{\rm loc}^{k+1,\alpha}$.
\end{theorem}

By suitably localizing this result,
in \cite{DFpartial} we could also prove the following:

\begin{theorem}
\label{cor:M}
Let $M$ be a smooth Riemannian manifold, and let $f,g:M\to \R^+$ be two continuous probability densities, locally bounded away from zero and infinity on $M$.
Let $T:M\to M$ denote the optimal transport map for the cost $c=d^2/2$ sending $f$ onto $g$.
Then there exist two closed sets $\Sigma_X,\Sigma_Y\subset M$ of measure zero such that
 $T:M\setminus \Sigma_X \to M\setminus \Sigma_Y$ is a homeomorphism of class $C_{\rm loc}^{0,\beta}$ for any $\beta<1$.
 In addition, if both $f$ and $g$ are of class $C^{k,\alpha}$, then
 $T:M\setminus \Sigma_X \to M\setminus \Sigma_Y$ is a diffeomorphism of class $C_{\rm loc}^{k+1,\alpha}$.
\end{theorem}

As discussed before, when the MTW condition fails the obstruction to the regularity of optimal maps is local, while in the quadratic case it was a global obstruction (namely the non-convexity of the target which do not allow to have a good control on the Monge-Amp\`ere measure of \(u\)). In the quadatic cost case, in   \cite{F-partial, FK-partial} it was shown that on a big set there is still a good  control on the Monge-Amp\`ere measure of the potential \(u\) allowing to apply the local theory of classical Monge-Amp\`ere equation, see the sketch of the proof of Theorem \ref{thm:partialMA}. However, the failure of the MTW condition does not allow us to use any local regularity estimate
for the PDE, therefore  a completely new strategy with respect to \cite{F-partial,FK-partial} has to be used.\\

The rough idea is the following: First of all recall that, thanks to Theorem \ref{cbrenier},
the optimal transport map $T$ is of the form $T_u(x)=\cexp_x(\nabla u(x))$
for some $c$-convex function $u$. Then, if $\bar x$ is a point where $u$
is twice differentiable (see Remark \ref{rmk:semiconv}), around that point $u$ looks like a parabola. In addition, by looking close enough to $\bar x$,
the cost function $c$ will be very close to the linear one and the densities will be almost constant there.
Hence $u$ is close to a convex function $v$ solving an optimal transport problem with linear cost and constant densities. In addition, since $u$ is close to a parabola, so is $v$. Hence, by Caffarelli's regularity theory, $v$ is smooth  and we can use this information to deduce that $u$ is even closer to a second parabola (given by the second order Taylor expansion of $v$ at \(\bar x\)) inside a small neighborhood.
By rescaling back this neighborhood at scale $1$ and iterating this construction, we obtain that $u$ is $C^{1,\beta}$ at \(\bar x\) for every $\beta \in (0,1)$.
Since this argument can be applied at every point in a neighborhood of \(\bar x\), we deduce that $u$ is $C^{1,\beta}$ there. This is summarized in the following:
\begin{proposition}\label{c1alpha}Let \(\mathcal C_{1}\) and \(\mathcal C_{2}\) be two closed sets satisfying
\begin{equation*}
B_{1/3}\subset \mathcal C_{1} ,\mathcal C_{2}\subset B_{3},
\end{equation*}
let \(f, g\) be 
 two densities supported in $\mathcal C_1$ and $\mathcal C_2$ respectively, and
let \(u: \mathcal C_1 \to \R\) be a \(c\)-convex function such that $\p_cu(\mathcal C_1)\subset B_3$
and \((T_{u})_\sharp (f\, dx) =g\,dy\). Then, for every \(\beta\in (0,1)\) there  exists a constant \(\eta_{0}>0\) such that the following holds: if 
$$ \|f-  \mathbf 1_{\mathcal C_1} \|_{\infty}+ \| g- \mathbf 1_{\mathcal C_2}\|_{\infty} 
 + \|c(x,y)+x\cdot y\|_{C^{2}(B_{3}\times B_{3})}
 + \left\|u-\frac{1}{2} |x|^{2}\right\|_{C^{0}(B_{3})} \leq \eta_{0},
 $$
 then \(u\in C^{1,\beta}(B_{1/6})\).
\end{proposition}
Once this result is proved, we know that $\partial^-u$ is a singleton at every point, so it follows from \eqref{eq:rel subdiff2} that
$$
\partial_cu(x) = \cexp_x (\partial^-u(x)).
$$
(The above identity is exactly what in general may fail for general $c$-convex functions, unless the MTW condition holds.)
Thanks to this fact we obtain that $u$ enjoys a comparison principle,
and this allows us to use a second approximation argument with solutions of the classical Monge-Amp\`ere equation
(in the spirit of \cite{Caf4}) to conclude that $u$ is $C^{2,\sigma'}$ in a smaller neighborhood, for some $\sigma'>0$.
Then higher regularity follows from standard elliptic estimates.
\begin{proposition}\label{C2alpha} Let \(u,f,g\) be as in Proposition \ref{c1alpha}, and assume in addition that \(c\in C^{k,\alpha}(B_{3}\times B_{3})\) and $f,g \in C^{k,\alpha}(B_{1/3})$
for some $k \geq 0$ and $\alpha \in (0,1)$.
There exists a small constant \(\eta_{1}\) such that if
$$
 \|f-  \mathbf 1_{\mathcal C_1} \|_{\infty}+ \| g- \mathbf 1_{\mathcal C_2}\|_{\infty} +
 \|c(x,y)+x\cdot y\|_{C^{2}(B_{3}\times B_{3})}+
\left\|u-\frac{1}{2} |x|^{2}\right\|_{C^{0}(B_{3})}\le \eta_{1},$$
then \(u \in C^{k+2,\alpha}(B_{1/6})\).
\end{proposition}
These results imply that $T_u$  is of class $C^{0,\beta}$ in neighborhood of $\bar x$
(resp. $T_u$ is of class $C^{k+1,\alpha}$ if $c \in C_{\rm loc}^{k+2,\alpha}$ and $f,g\in C_{\rm loc}^{k,\alpha}$).
Being our assumptions completely symmetric in $x$ and $y$, we can apply the same argument to the optimal map $T^*$ sending $g$ onto $f$.
Since $T^*=(T_u)^{-1}$, it follows that $T_u$ is a global
homeomorphism of class $C_{\rm loc}^{0,\beta}$  (resp. $T_u$ is a global diffeomorphism of class $C_{\rm loc}^{k+1,\alpha}$) outside a closed set of measure zero.

\medskip    \section{Open problems and further perspectives}

In this last section we briefly describe some further material related to Monge-Amp\`ere type equations and state some  open problems.

\medskip
\subsection{General prescribed Jacobian equations}\label{PJsec} Equations \eqref{eq:MA} and \eqref{eq:MAcost} can be seen has particular cases of \emph{prescribed Jacobian equations} of the following form:
\begin{equation}\label{eq:PJ}
\det \bigl(\nabla [T(x, u, \n u)]\bigr)=\psi (x, u, \n u),
\end{equation}
where \(T=T(x,z,p): \Omega\times \mathbb R\times \mathbb R^n\to \mathbb R^n\). In case \(\det( \nabla_p T)\ne 0 \), arguing as in Section \ref{MTWsec}
one sees that \eqref{eq:PJ} becomes
\begin{equation}\label{eq:MAgen}
\det\bigl(D^2u - \mathcal A(x,u,\n u)\bigr)=f(x,u,\n u).
\end{equation}
Let us notice that the classical Monge-Amp\`ere equation corresponds to the case \(T(x,z,p)=p\), and more generally the optimal transport case described in this note
corresponds to the case 
\[
\nabla_x c (x, T(x,p))=-p,
\]
that is \(T(x,p)=\cexp_x(p)\) (compare with Section \ref{gcsect}). 

Motivated by problems arising in geometric optics, Trudinger began in \cite{T,T1} a systematic  study of equations of the form \eqref{eq:PJ}-\eqref{eq:MAgen} in the particular when \(T\) is obtained through a generating function \(G:\mathbb R^n\times \mathbb R^n\times \mathbb R \). Let us briefly present this theory, referring to \cite{T1} for more details.

Given \(G\) as above, a function \(u\) is said \(G\)-convex on \(\Omega\) if for every \(x_0\in \Omega\) there exist \(y_0\) and \(z_0\) such that
\begin{equation}
\label{eq:uG}
u(x_0)= -G(x_0,y_0,z_0)
\quad \text{and}\quad u(x)\ge - G(x,y_0,z_0)\quad \forall \,x\in \Omega.
\end{equation}
Then, under suitable assumptions on $G$ one can define the maps $T=T(x,z,p)$ and $Z=Z(x,z,p)$
through the relations
\[
\nabla_xG(x, T(x,u,p),Z(x,u,p))=-p,\qquad -G(x, T(x,u,p),Z(x,u,p))=u.
\]
With this choice one gets from \eqref{eq:uG} that
\(y_0=T(x_0,u(x_0), Du(x_0))\) whenever $u \in C^1$.
Note that the optimal transportation case corresponds to the choice \(G(x,y,z):=c(x,y)-z\).

Then, under some necessary structural conditions on $G$
(which are the analogous of the MTW condition in this context), in \cite{T,T1}
Trudinger started developing a theory parallel to the one described in Section \ref{MTWsec}.
It would be interesting to extend all the results valid in the optimal transportation case
to this general setting. Also, in case \(G\) does not satisfy such necessary structural conditions,
it would be nice to check whether an analogue of Theorem \ref{thm:c} still holds.

%

\medskip

\subsection{Open Problems}\label{OPsec}We list some open problems related to optimal transportation and the regularity theory for general Monge-Amp\`ere type equations.\\

\begin{enumerate}
\item As we already mentioned, in dimension $2$ stronger
regularity results for solutions of \eqref{eq:MAclassical} are available \cite{AL2,Caf-CPDE,FL}. In particular, Alexandrov showed in \cite{AL2} that \(u\) is continuous differentiable assuming only
the upper bound \(\det D^2u \le \lambda\).
Hence, in relation to Theorem \ref{w21eps} a natural question becomes: is it
possible to prove \(W^{2,1}_{\rm loc}\)  (or even $W^{2,1+\e}_{\rm loc}$) regularity of \(u\) in the \(2\)-d case assuming only an upper bound on
$\det D^2u$? Apart from its own interest, such a result would have applications in relation to
extend Theorem \ref{thm:mainSG} outside of the periodic setting.\\

\item In \(\R^2\) there is a link between the classical Monge-Amp\`ere equation \eqref{eq:MAclassical} and the theory of \emph{quasi-conformal} maps. Indeed, if \(u\) is a smooth solution of \eqref{eq:MAclassical} we can consider consider the maps \(\Phi_1, \Phi_2: \Omega \to \R^2\) defined as 
\[
\Phi_1 (x,y):=(\partial_x u(x,y),y)\qquad \Phi_2(x,y):=(x,\partial_y u(x,y)).
\]
Then, if \(\lambda \le \det D^2 u\le1/ \lambda\), a simple computation shows that \(\Phi:=\Phi_1\circ (\Phi_2)^{-1}\) is \(1/\lambda^2\)-quasiconformal, that is
\[
\lambda^2 \|\nabla \Phi\|_{HS}^2 \le 2  \det \nabla \Phi
\]
where \(\|\cdot\|_{HS}\) is the Hilbert-Schmidt norm.
 In view of the known higher-integrability theory for quasi-conformal maps \cite[Theorem 13.2.1]{AGI}, one is led to conjecture that in dimension $2$ the sharp version of Theorem \ref{w21eps} should be
\[
D^2 u \in L_{w,{\rm loc}}^{\frac {1+\lambda^2}{1-\lambda^2}}(\Omega)\qquad\text{
provided } \lambda\le \det D^2 u\le 1/\lambda \text{ inside }\Omega,
\] 
that is
$$
\sup_{s>0} s^{\frac {1+\lambda^2}{1-\lambda^2}}\,\bigl|\bigl\{|D^2u|>s\bigr\}\cap \Omega'\bigr| \leq C \qquad \forall\,\Omega'\Subset \Omega.
$$
In particular one would get that $D^2u \in L^p_{\rm loc}(\Omega)$ for all $p<\frac {1+\lambda^2}{1-\lambda^2}$.\\

\item As we explained at the end of Section \ref{semigeo}, in the case of smooth initial data existence of \emph{smooth} solutions to the semigeostrophic equations is known only for short time. In analogy with the 2-d incompressible Euler
equations (see for instance \cite{LoeSG} for a discussion on the analogy between these two equations), it 
would be extremely interesting to understand if, at least in the two dimensional periodic case, there are conditions on the initial data which ensure global in time existence of smooth solutions.\\

\item 
The proof of the fact that the MTW condition holds on perturbations of the round sphere
is extremely delicate and relies on the fact that, in the expression of the MTW tensor
\eqref{eq:MTW tensor}, 15 terms ``magically'' combine to form a ``perfect square''
(see \cite[page 127]{FRV}), allowing the authors to prove that the MTW tensor satisfies the right inequality. It seems unlikely to us that this is just a coincidence, and we believe that this fact should
be a sign of the presence of an underlying and deeper structure which has not been yet found.\footnote{To explain this with an example, one may think at covariant derivatives in classical Riemannian geometry: if instead of using them one just uses standard differentiation in charts,
one would end up with complicated expressions which, by ``magic'', have a lot of simplifications.
On the other hand, by using covariant derivatives, formulas automatically simplify.}\\

\item As mentioned in Section \ref{sect:Riem}, the stay-away property
\eqref{eq:stayaway}
is crucial to be able to localize Theorems \ref{thm:C1a} and \ref{thm:C1a mtw0} to the setting
of Riemannian manifolds. In addition, apart from this application,
proving (or disproving) the validity of this property would represent an important
step in the understanding of the geometry of optimal transportation. \\

\item  As explained in Section \ref{sect:cut}, there is a formal argument
which suggests that the $\MTW$ condition should
imply that all  injectivity domains $\I(x)$ are convex for
every $x \in M$. Proving this result in full generality would be
interesting both for the regularity of optimal transport maps and as a purely geometric result.\\

\item  Another natural step in understanding the relation between the MTW condition
and the regularity of optimal map would be to prove that 
the MTW condition
and the convexity of the tangent cut-loci are both necessary and sufficient conditions for the 
continuity of optimal transport maps on Riemannian manifolds (with the cost given by
the squared distance).
As mentioned at the end of Section \ref{sect:cut}, only the necessity is currently known in general,
while the sufficiency is known only in dimension $2$. A nice problem would be to prove the sufficiency
in every dimension. Notice that, if one could also prove (4) above, then as a corollary one would obtain that MTW is necessary and sufficient for regularity of optimal maps on any Riemannian manifold.\\

\item It would be very interesting to give estimates on the Hausdorff dimension of the singular sets \(\Sigma_X\) and \(\Sigma_Y\) appearing in the partial regularity Theorems \ref{thm:c} and \ref{cor:M}. In view of all known example, one is led to conjecture that \({\rm dim}_{\mathcal H} (\Sigma_X),{\rm dim}_{\mathcal H} (\Sigma_Y) \le n-1\). This is not known even in the case of the quadratic cost when the target is not convex, although some partial results in this directions have been
obtained in \cite{F-partial}.  A even more difficult problem would be to prove that the singular sets are always rectifiable or even, under suitable assumption on the densities, smooth.\\

\item Parabolic versions of the Monge-Amp\`ere equation naturally appear in geometric evolution problems (see for instance \cite{DS} and the references therein).
More recently, in \cite{Kit,KSW} the authors studied the parabolic version of \eqref{eq:MAgeneral} 
in the context of optimal transport and they showed that, 
if the MTW condition holds, then
under suitable conditions on the initial and target domains
the solution exists globally in time 
and converges exponentially fast, as $t \to \infty$ to the solution of the original mass transport problem.
Apart from its own interest, this result provides a potential way to numerically solve the optimal transport problem by taking a smooth initial condition and running the parabolic flow for sufficiently long time.
It would be extremely interesting to understand
if some similar results may hold (in some weak sense) even without assuming the validity of the MTW condition, and then obtain partial regularity results in the spirit of Theorem \ref{thm:c}.\\

\item The assumption of the existence of a generating function $G$ ensures that the matrix
$\mathcal A(x,u,\nabla u)$ appearing in \eqref{eq:MAgen} is symmetric.
However, for the general class of equations \eqref{eq:MAgen} arising from the prescribed Jacobian equations \eqref{eq:PJ}, 
there is no reason why the matrix $\mathcal A$ should be symmetric.
It would be very interesting to
understand under which structural assumptions on $\mathcal A$
one can develop a regularity theory (see \cite[Corollary 1.2]{T}
for a result in dimension 2).
\end{enumerate}

\medskip

\textit{Acknowledgments:}
We wish to thank Neil Trudinger for several useful comments
on a preliminary version of this paper.
The second author is partially supported by NSF Grant DMS-1262411.
Both authors acknowledge the support of the ERC ADG Grant GeMeThNES.

%

\end{document}